\documentclass[11pt,reqno]{amsart}
\usepackage{amsmath,amscd,amssymb,amsfonts,amsthm,courier,relsize,bm,tikz}
\usetikzlibrary{patterns}
\usepackage{hyperref,enumitem,mathrsfs,slashed,mathtools,graphicx}

\usepackage{xcolor}  
\hypersetup{
    colorlinks,
    linkcolor={red!50!black},
    citecolor={blue!70!black},
    urlcolor={blue!80!black}
}

\usepackage[cmtip,matrix,arrow]{xy}

\newtheorem{theorem}{Theorem}[section]
\newtheorem{corollary}[theorem]{Corollary}
\newtheorem{proposition}[theorem]{Proposition}

\newtheorem{lemma}[theorem]{Lemma}
\theoremstyle{definition}    
\newtheorem{definition}[theorem]{Definition}
\theoremstyle{remark}

\newtheorem{remark}[theorem]{Remark}

\newtheorem{example}[theorem]{Example}

\newcommand{\pair}[2]{\langle #1, #2 \rangle}
\newcommand{\ignore}[1]{}

\newcommand{\ol}[1]{\overline{#1}}

\newcommand{\ti}[1]{\widetilde{#1}}
\newcommand{\wh}[1]{\widehat{#1}}
\newcommand{\st}[1]{\mathsf{#1}}
\newcommand{\scr}[1]{\mathscr{#1}}

\newcommand{\tn}[1]{\textnormal{#1}}
\renewcommand{\i}{{\mathrm{i}}}

\def\Ad{\ensuremath{\textnormal{Ad}}}

\def\g{\ensuremath{\mathfrak{g}}}
\def\k{\ensuremath{\mathfrak{k}}}
\def\t{\ensuremath{\mathfrak{t}}}
\def\h{\ensuremath{\mathfrak{h}}}

\def\n{\ensuremath{\mathfrak{n}}}

\def\hvee{\ensuremath{\textnormal{h}^\vee}}

\def\L{\ensuremath{\mathcal{L}}}
\def\O{\ensuremath{\mathcal{O}}}
\def\R{\ensuremath{\mathcal{R}}}

\def\T{\ensuremath{\mathcal{T}}}

\def\E{\ensuremath{\mathcal{E}}}
\def\B{\ensuremath{\mathcal{B}}}

\def\P{\ensuremath{\mathcal{P}}}

\def\Y{\ensuremath{\mathcal{Y}}}

\def\K{\ensuremath{\mathcal{K}}}
\def\X{\ensuremath{\mathcal{X}}}
\def\J{\ensuremath{\mathcal{J}}}
\def\O{\ensuremath{\mathcal{O}}}
\def\M{\ensuremath{\mathcal{M}}}

\def\c{\ensuremath{\mathsf{c}}}

\def\dirac{\ensuremath{\slashed{\partial}}} 

\def\sY{\ensuremath{\mathsmaller{\mathcal{Y}}}}

\def\bC{\ensuremath{\mathbb{C}}}
\def\bR{\ensuremath{\mathbb{R}}}
\def\bZ{\ensuremath{\mathbb{Z}}}
\def\bN{\ensuremath{\mathbb{N}}}

\def\bB{\ensuremath{\mathbb{B}}}
\def\bK{\ensuremath{\mathbb{K}}}

\def\End{\ensuremath{\textnormal{End}}}
\def\Aut{\ensuremath{\textnormal{Aut}}}

\def\Hom{\ensuremath{\textnormal{Hom}}}

\def\ran{\ensuremath{\textnormal{ran}}}
\def\.{\ensuremath{\cdot}}
\def\ker{\ensuremath{\textnormal{ker}}}

\def\supp{\ensuremath{\textnormal{supp}}}

\def\id{\ensuremath{\textnormal{id}}}

\def\Cl{\ensuremath{\textnormal{Cl}}}
\def\Cliff{\ensuremath{\textnormal{Cliff}}}
\def\index{\ensuremath{\textnormal{index}}}
\def\vol{\ensuremath{\textnormal{vol}}}

\def\dim{\ensuremath{\textnormal{dim}}}

\def\aff{\ensuremath{\textnormal{aff}}}
\def\KK{\ensuremath{\textnormal{KK}}}

\def\Sym{\ensuremath{\textnormal{Sym}}}
\def\pt{\ensuremath{\textnormal{pt}}}

\def\dom{\ensuremath{\textnormal{dom}}}

\def\K{\ensuremath{\textnormal{K}}}
\def\Bott{\ensuremath{\textnormal{Bott}}}
\def\Cyl{\ensuremath{\textnormal{Cyl}}}

\def\wt{\ensuremath{\textnormal{wt}}}

\def\rot{\ensuremath{\tn{rot}}}
\def\Irr{\ensuremath{\tn{Irr}}}

\def\hGamma{\ensuremath{\widehat{\Gamma}}}

\def\hgamma{\ensuremath{\widehat{\gamma}}}
\def\hLambda{\ensuremath{\widehat{\Lambda}}}

\begin{document}
\sloppy
\title{Quantization of Hamiltonian loop group spaces}
\author{Yiannis Loizides}
\author{Yanli Song}

\begin{abstract}
We prove a Fredholm property for spin-c Dirac operators $\st{D}$ on non-compact manifolds satisfying a certain condition with respect to the action of a semi-direct product group $K\ltimes \Gamma$, with $K$ compact and $\Gamma$ discrete.  We apply this result to an example coming from the theory of Hamiltonian loop group spaces.  In this context we prove that a certain index pairing $[\X] \cap [\st{D}]$ yields an element of the formal completion $R^{-\infty}(T)$ of the representation ring of a maximal torus $T \subset H$; the resulting element has an additional antisymmetry property under the action of the affine Weyl group, indicating $[\X] \cap [\st{D}]$ corresponds to an element of the ring of projective positive energy representations of the loop group.
\end{abstract}
\maketitle
\vspace{-0.5cm}

\section{Introduction}
Let $H$ be a compact, connected Lie group, and let $LH$ denote the loop group.  Let $(\M,\omega_{\M},\Phi_{\M})$ be a Hamiltonian $LH$-space, with proper moment map $\Phi_{\M} \colon \M \rightarrow L\h^\ast$.  Many well-known results for Hamiltonian $H$-spaces have parallels for Hamiltonian $LH$-spaces.  Examples include the convexity theorem \cite{MWVerlindeFactorization}, the cross-section theorem \cite{MWVerlindeFactorization}, and Kirwan surjectivity \cite{BottTolmanWeitsman}.  In this paper we describe and study an index-theoretic `quantization' of $\M$, in the spirit of the `quantization' of Hamiltonian $H$-spaces via the equivariant index of twisted Dirac operators, and the ${[Q,R]=0}$ Theorem (cf. \cite{GuilleminSternbergConjecture,MeinrenkenSymplecticSurgery,TianZhang,
ParadanRiemannRoch,ParadanVergneSpinc,HochsSongSpinc}).

Given a prequantum line bundle $L$ over $\M$ which is equivariant for a suitable $U(1)$ central extension $\wh{LH}$ of the loop group, one expects to be able to associate to $(\M,L)$ a \emph{positive energy representation} $Q(\M,L)$ of $\wh{LH}$, or more generally a formal difference of such representations.  The map $(\M,L) \mapsto Q(\M,L)$ is expected to have several desirable properties; for example, an integral affine coadjoint orbit should be sent to the corresponding irreducible positive energy representation.  By analogy with finite-dimensional Hamiltonian $H$-spaces, one might like to define $Q(\M,L)$ by constructing a suitable Dirac operator on $\M$, twisting by $L$, and then taking the index.  Since $\M$ is infinite dimensional, it is unclear how one would make sense of this procedure. 

One alternative was explored in \cite{MeinrenkenKHomology}.  The gauge action of the based loop group $\Omega H \subset LH$ on $\M$ is free and proper.  The quotient $M=\M/\Omega H$ is a smooth finite-dimensional compact $H$-manifold, known as a \emph{quasi-Hamiltonian} $H$-space \cite{AlekseevMalkinMeinrenken}.  In fact $(\M,\omega_{\M},\Phi_{\M})$ can be recovered from suitable data on $M$, giving a 1-1 correspondence between Hamiltonian loop group spaces and quasi-Hamiltonian spaces.  One might declare the `quantization' of $\M$ to be a suitable `quantization' of the finite-dimensional space $M$.  There are natural examples of quasi-Hamiltonian spaces that do not possess any spin-c structure, and so a suitable Dirac operator is lacking.  It is possible to get around this obstacle by using \emph{twisted} K-homology; this was done in \cite{MeinrenkenKHomology}, and the connection with loop group representations was made via the Freed-Hopkins-Teleman Theorem.  A $[Q,R]=0$ theorem was also proved in this context \cite{AMWVerlinde} using symplectic cutting techniques. 

In this paper we explore another option, involving the equivariant index of an operator on a finite-dimensional submanifold of $\M$.  In a previous paper with E. Meinrenken \cite{LMSspinor} we constructed a suitable finite-dimensional `global transversal' $\Y \subset \M$, as well as a canonical spinor module $S \rightarrow \Y$.  In this paper we prove that the corresponding spin-c Dirac operator $\st{D}$ acting on sections of $S$ twisted by $L$ defines an element $[\st{D}]$ in a suitable K-homology group.  Taking the `index pairing' or `cap product' with a suitable K-cohomology class $\st{x}$, we obtain an element $\st{x} \cap [\st{D}]$ of the formal completion of the representation ring $R^{-\infty}(T)$ of a maximal torus $T \subset H$.  The corresponding multiplicity function is anti-symmetric under the action of the affine Weyl group, indicating that it is the numerator of the Weyl-Kac character formula for a (graded) positive energy representation---and we take this to be our definition of $Q(\M,L)$.

In another article \cite{LoizidesGeomKHom} the first author proved that $Q(\M,L)$ coincides with the image under the Freed-Hopkins-Teleman isomorphism of the quantization of $M=\M/\Omega H$ defined in terms of twisted K-homology, showing that these two definitions are consistent with each other.  We discuss this relationship briefly in Section \ref{ssec:OtherApproaches}.  There is a third approach \cite{SongDiracLoopGroup}, due to the second author, that we hope to return to in the future.

As motivation for our definition, let us briefly discuss the analogous construction in finite dimensions, where it is equivalent to the usual definition (see also \cite[Section 6]{LMSspinor}).  Let $\mu \colon N \rightarrow \h^\ast$ be a compact Hamiltonian $H$-space with prequantum line bundle $L$.  Choosing a compatible almost complex structure, one obtains a spinor module $S=\wedge T^\ast_{0,1} N \otimes L$ and spin-c Dirac operator $\dirac^S$.  The $H$-equivariant index $Q(N,L)=\index(\dirac^S)$ is an element of the representation ring $R(H)$.  Let
\[ \Delta(t)=\prod_{\alpha>0} (1-t^{-\alpha}) \]
be the Weyl denominator, the (super) character of the $\bZ_2$-graded representation $\wedge \n_-$ of $T$, where $\n_-$ is the sum of the negative root spaces.  The $T$-equivariant index of $\dirac^S$ twisted by the trivial bundle with fibres $\wedge \n_-$ is
\begin{equation} 
\label{eqn:DiracN}
\index(\dirac^S\otimes \wedge \n_-)=\Delta \cdot Q(N,L)|_T \in R(T);
\end{equation}
it may also be viewed as an \emph{index pairing} $\pair{\dirac^S}{\wedge \n_-}$ between a K-homology class $[\dirac^S]$ and a K-theory class $[\wedge \n_-]$.

The Bott element $\Bott(\n_-)$ for $\h^\ast/\t^\ast$ is an element of the K-theory group $\K^0_T(\h^\ast/\t^\ast)$, which may be described in terms of a bundle morphism
\[ (\h^\ast/\t^\ast) \times \wedge^{\tn{even}} \n_- \rightarrow (\h^\ast/\t^\ast) \times \wedge^{\tn{odd}} \n_-\]
invertible everywhere except at the origin.  Let $q$ be the composition of the moment map $N \rightarrow \h^\ast$ with the projection $\h^\ast \rightarrow \h^\ast/\t^\ast$.  Since $N$ is compact, we may replace $\wedge \n_-$ in \eqref{eqn:DiracN} with the pullback of the Bott element, giving the index pairing $\pair{\dirac^S}{q^\ast\Bott(\n_-)}$.  As $\Bott(\n_-)$ is supported at $0 \in \h^\ast/\t^\ast$, we may restrict to the pre-image $\Y=q^{-1}(B)$, where $B$ is a small ball in $\h^\ast/\t^\ast$ containing the origin.  Thus
\begin{equation} 
\label{eqn:AltDef}
\Delta \cdot Q(N,L)|_T=\pair{\dirac^S_\Y}{q^\ast \Bott(\n_-)} 
\end{equation}
where $\dirac^S_\Y$ denotes the restriction of $\dirac^S$ to the open subset $\Y$, and the right hand side is an index pairing on $\Y$.  In the special case that $\mu$ is transverse to $\t^\ast$ (i.e. $0$ is regular value of $q$), the fibre $\X=\mu^{-1}(\t^\ast)=q^{-1}(0)$ is a smooth submanifold with trivial normal bundle isomorphic to $\X \times (\h/\t)$.  Using the spinor module $\wedge \n_+$ for $\Cliff(\h/\t)$, by the 2-out-of-3 property we obtain a spinor module $S_{\X}=\Hom_{\Cliff(\h/\t)}(\wedge \n_+,S)$ for $\X$, and the index pairing on the right hand side of \eqref{eqn:AltDef} equals the usual index of a Dirac operator for $S_{\X}$.

Turning this around, we see that we can \emph{define} $Q(N,L)$ as the unique element of $R(H)$ satisfying \eqref{eqn:AltDef}; equivalently, the index pairing in \eqref{eqn:AltDef} gives the Weyl numerator of $Q(N,L)$.  We will see that a partial analogue of this definition works well for Hamiltonian loop group spaces $\Phi_{\M}\colon \M \rightarrow L\h^\ast$.  In this new context, $\Y$ will be a finite dimensional submanifold of $\M$, playing the role of a small `thickening' of the possibly singular subset $\X=\Phi_{\M}^{-1}(\t^\ast)$, analogous to its role in the finite dimensional setting.  We will define the quantization of $(\M,L)$ as the element of the fusion ring whose Weyl-Kac numerator is given by an index pairing on $\Y$ as on the right hand side of \eqref{eqn:AltDef}.  In the special case $\Phi_{\M}$ is transverse to $\t^\ast \subset L\h^\ast$, we can again replace the index pairing with the index of a Dirac operator on $\X$.

One attractive feature of our definition is that it is amenable to study with the `Witten deformation' (cf. \cite{TianZhang,ParadanRiemannRoch,WittenNonAbelian}).  In a sequel to this paper \cite{LSWittenDef} we study this deformation in detail, and obtain a formula in the spirit of Paradan \cite{ParadanRiemannRoch} for the index pairing:
\begin{equation} 
\label{eqn:NormSqrFormula}
\st{x}\cap [\st{D}]=\sum_{\beta \in W\cdot \B} \index(\sigma_{\beta,\st{x}} \otimes \Sym(\nu_\beta)),
\end{equation}
where $\st{x}$ is a suitable K-theory class, $\B \subset \t_+$ indexes components of the critical set of the norm-square of the moment map, $\sigma_{\beta,\st{x}}$ is a transversally elliptic symbol on the fixed-point set $\Y^\beta$, and $\nu_\beta$ is the normal bundle to $\Y^\beta$ in $\Y$ equipped with a `$\beta$-polarized' complex structure.  An interesting feature is that the index set $\B$ is infinite, and \eqref{eqn:NormSqrFormula} contains infinitely many non-zero terms.  Analogous to \cite{ParadanRiemannRoch}, \eqref{eqn:NormSqrFormula} leads to a new proof of the $[Q,R]=0$ Theorem for Hamiltonian loop group spaces.  The formula \eqref{eqn:NormSqrFormula} is the K-theoretic analogue of a formula for (twisted) Duistermaat-Heckman distributions studied in \cite{DHNormSquare}.

The main step in proving that $[\st{D}]$ defines a suitable K-homology class involves an interesting interplay between two group actions, one compact and one discrete, on a non-compact manifold.  This argument can be carried out more generally, and we do this in Section \ref{sec:FredholmProperty}.

Throughout we use the language and basic techniques of analytic K-homology, and, in a few places, Kasparov's KK-theory.  We provide a brief introduction to some aspects of KK-theory in Section \ref{sec:KHomologyPrelim}.  We have also included two short appendices.  The first describes the interaction between group automorphisms and the descent map in $\KK$-theory.  This discussion is used to prove the anti-symmetry properties of the index pairing.  The second appendix describes a common generalization of two well-known consequences of the Rellich Lemma.  This result is used in determining a cycle representing the intersection product.

\vspace{0.3cm}

\noindent \textbf{Acknowledgements.} We especially thank Eckhard Meinrenken for many helpful discussions and encouragement, and for providing feedback on an earlier draft.  We also thank Nigel Higson for helpful discussions.  Y. Song is supported by NSF grant 1800667.

\vspace{0.3cm}

\noindent \textbf{Notation.}
We often deal with vector spaces/bundles that are $\bZ_2$-graded.  In this context, $[a,b]$ will denote the \emph{graded commutator} of the linear operators $a$, $b$.  We use the Koszul sign rule for tensor products.

For a finite-dimensional real Euclidean vector space $(E,g)$, the Clifford algebra $\Cl(E)$ denotes the complex $\bZ_2$-graded algebra generated by odd elements $e \in E$ subject to relations
\[ [e_1,e_2]=e_1e_2+e_2e_1=-2g(e_1,e_2), \qquad \forall e_1,e_2 \in E.\]
If $E \rightarrow Y$ is a Euclidean vector bundle, $\Cl(E)$ is the bundle with fibres $\Cl(E_y)$, $y \in Y$.

For a Hilbert space $H$ the inner product (resp. norm) will be denoted by $(\cdot,\cdot)$ (resp. $\|\cdot \|$).  The $C^\ast$ algebras of bounded (resp. compact) operators on $H$ will be denoted $\bB(H)$ (resp. $\bK(H)$). 

For a Riemannian manifold $(Y,g)$, we sometime use the metric $g$ to identify $TY \simeq T^\ast Y$.  If $E \rightarrow Y$ is a Hermitian vector bundle, the space of $L^2$ sections, denoted $L^2(Y,E)$, is a Hilbert space equipped with the inner product
\[ (u,v)=\int_Y \pair{u}{v} d\vol,\]
where $\pair{\cdot}{\cdot}$ denotes the Hermitian inner product on the fibres, and $d\vol$ is the Riemannian volume.  The corresponding norm will be denoted $\|\cdot \|$, while the point-wise norm will be denoted $|\cdot |$, hence
\[ \|u\|^2=\int_Y |u|^2 d\vol.\]

If $K$ is a compact Lie group with Lie algebra $\k$, we write $\Irr(K)$ for the set of isomorphism classes of irreducible representations of $K$, and $R(K)$ for the representation ring.  The formal completion $R^{-\infty}(K)$ of $R(K)$ consists of formal infinite linear combinations of irreducibles $\pi \in \Irr(K)$ with coefficients in $\bZ$.

Let $G$ be a Lie group acting smoothly on a manifold $Y$.  For $g \in G$, $y \in Y$ we write $g.y$ for the action of $g$ on $y$.  For $\xi \in \g$ the vector field $\xi_Y$ on $Y$ is defined by
\[ \xi_Y(y)=\frac{d}{dt}\bigg|_0\exp(-t\xi).y.\]
The map $\xi \mapsto \xi_Y$ is a Lie algebra homomorphism.  Let $E \rightarrow Y$ be a $G$-equivariant vector bundle.  Then $g \in G$ acts on a section $e$ of $E$ by
\[ (g \cdot e)(y)=g.e(g^{-1}.y).\]

\section{$\KK$-theory and crossed products} \label{sec:KHomologyPrelim}
In this section we give a brief introduction to KK-theory and crossed products.  For readers unfamiliar with KK-theory, we should mention that we use only some of the more basic aspects, applied to geometrically-motivated examples of $C^\ast$-algebras.  For most of what we do, analytic K-homology (as described, for example, in the book of Higson and Roe\cite{HigsonRoe}) suffices; KK-cycles appear in a few places as a convenient means of defining index pairings.  In order for various constructions in $\KK$-theory to be well-behaved, one often requires that the two arguments are separable $C^\ast$-algebras; we will always assume this without mention below.  References for $\KK$-theory include \cite{Blackadar, KasparovNovikov, HigsonPrimer}. 

\subsection{$\KK$-theory.}
A (right) \emph{Hilbert} $B$-\emph{module} is a right $B$-module equipped with a $B$-sesquilinear map
\[ (\cdot,\cdot)_B \colon \E \times \E \rightarrow B \] 
called the $B$-\emph{valued inner product}, satisfying properties analogous to the properties of an inner product on a Hilbert space but with the field of scalars $\bC$ replaced with the $C^\ast$-algebra $B$, and such that $\E$ is complete with respect to the norm $\|e\|_{\E}=\|(e,e)_B\|_B^{1/2}$. 

By analogy with Hilbert spaces, one defines a $C^\ast$ algebra $\bB_B(\E)$ consisting of ($B$-linear) transformations which are `adjointable' relative to $(\cdot,\cdot)_B$, and a subalgebra $\bK_B(\E)$ as the closure of the linear span of the `finite rank' transformations $\theta_{e_1,e_2}\colon e \mapsto e_1\cdot (e_2,e)_B$, for $e_1,e_2 \in \E$.

When $B=\bC$, a Hilbert $B$-module is nothing but a Hilbert space.  The other prototypical example which we shall use is for the algebra $B=C_0(Y)$ of continuous functions vanishing at infinity on a locally compact space $Y$.  Let $E$ be a Hermitian vector bundle over $Y$.  The space $\E=C_0(Y,E)$ of continuous sections of $E$ vanishing at infinity is a Hilbert $C_0(Y)$-module, with the $C_0(Y)$-valued inner product given by the Hermitian structure.  In this case $\bB_B(\E)$ (resp. $\bK_B(\E)$) consists of continuous sections of $\End(E)$ which are bounded (resp. vanish at infinity).

Cycles for $\KK(A,B)$ are triples $(\E,\rho,F)$ consisting of a countably generated $\bZ_2$-graded Hilbert $B$-module, a graded $^\ast$-representation 
\[ \rho \colon A \rightarrow \bB_B(\E) \]
and an odd operator $F \in \bB_B(\E)$ such that for all $a \in A$
\begin{equation}
\label{ConditionOnF} 
\rho(a)(F-F^\ast), \quad \rho(a)(1-F^2), \quad [\rho(a),F] \in \bK_B(\E). 
\end{equation}
We often drop $\rho$ from the notation, as it is usually clear from the context.  

Cycles may be added by taking the direct sum of the Hilbert modules, representations, and operators.  A \emph{homotopy} between two cycles $(\E_i,\rho_i,F_i)$, $i=0,1$ is a cycle $(\E,\rho,F)$ for the pair $\big(A,B \otimes C([0,1])\big)$ such that the evaluation homomorphisms for $0,1 \in [0,1]$ recover $(\E_0,\rho_0,F_0)$, $(\E_1,\rho_1,F_1)$ respectively.  Homotopy defines an equivalence relation on the set of cycles.  This equivalence relation is generated by three simpler instances of homotopy of cycles\footnote{This is true if $A$, $B$ are separable, as we are assuming.}: (1) \emph{unitary equivalence}, given by an isomorphism $\E_1 \rightarrow \E_2$ intertwining all structures, (2) addition of \emph{degenerate cycles}, for which all three operators in \eqref{ConditionOnF} are $0$ for all $a \in A$, and (3) \emph{operator homotopy}, given by a family $(\E,\rho,F_t)$ of cycles, where $t \in [0,1] \mapsto F_t \in \bB_B(\E)$ is norm-continuous.

The abelian group $\KK(A,B)$ is the set of homotopy classes of cycles.  $\KK(A,B)$ is contravariant in $A$ and covariant in $B$, hence is a `bi-functor' to the category of abelian groups.  The special case $\K^0(A):=\KK(A,\bC)$ is the $\K$-\emph{homology} of $A$, while $\K_0(B):=\KK(\bC,B)$ is the $\K$-\emph{theory} of $B$.  $\KK$ is finitely additive in both arguments, and in fact
\begin{equation} 
\label{eqn:KHomologyAdditivity}
\KK(\oplus_i A_i,B)=\prod_i \KK(A_i,B) 
\end{equation}
for a countable direct sum.\footnote{Under some conditions on $A$ (in particular if $A=\bC$), $\KK(A,-)$ is countably additive, cf. \cite{Blackadar} Section 23.15.}

For $G$ a locally compact group and $G$-$C^\ast$ algebras $A$, $B$ there is a $G$-equivariant $\KK$-group denoted $\KK_G(A,B)$.  Its cycles consist of triples $(\E,\rho,F)$ similar to before, together with a continuous, isometric $G$-action on $\E$ compatible with the bimodule structure, such that $g \mapsto \Ad_g(F)$ is norm continuous and $F$ `almost commutes' with the $G$-action in the sense that $\rho(a)(\Ad_g(F)-F) \in \bK_B(\E)$ for all $g \in G$, $a \in A$.  If $G$ is compact, then by averaging one can assume $F$ commutes with the $G$-action.

One of the main features of $\KK$-theory is the intersection product
\[ \otimes \colon \KK(A,B) \times \KK(B,C) \rightarrow \KK(A,C), \]
which systematically generalizes a number of constructions in K-theory, including pull-backs, wrong-way maps, and index pairings of K-homology elements with K-theory elements.  In a couple of places in this paper we will use the intersection product in a mild way.  Given cycles $x_i=(\E_i,\rho_i,F_i)$, $i=1,2$ for the pairs $(A,B)$, $(B,C)$ respectively, there is a cycle $(\E,\rho,F)$ for the intersection product $[x_1]\otimes [x_2]$ with
\[ \E=\E_1 \wh{\otimes}_{\scriptscriptstyle B} \E_2, \qquad \rho=\rho_1\otimes 1 \]
where $\wh{\otimes}_B$ is a completion of the algebraic graded tensor product of $B$-modules (after possibly quotienting out by vectors of length $0$).  The operator $F$ is not uniquely determined by $F_1$, $F_2$, but there is a criterion for verifying that a given $F$ represents the intersection product.

\subsection{Unbounded cycles.}
It is sometimes easier to work with `unbounded' cycles, as introduced by Baaj-Julg \cite{BaajJulg} and further developed by Kucerovsky \cite{KucerovskyUnbounded}.  Additional references include \cite{Blackadar} and \cite[Appendix]{MislinValette}.

Unbounded self-adjoint operators on a Hilbert $B$-module are defined in a way analogous to the case of Hilbert spaces.  One important difference between Hilbert spaces and more general Hilbert $B$-modules is that Hilbert $B$-submodules need not have complements in general.  An unbounded self-adjoint operator on a Hilbert module $\E$ is called \emph{regular} if the orthogonal space to its graph is a complementary Hilbert submodule in $\E \oplus \E$.
\begin{definition}[\cite{BaajJulg}]
\label{def:Unbounded}
An \emph{unbounded} cycle for $\KK(A,B)$ is a triple $(\E,\rho,\st{D})$ consisting of a $\bZ_2$-graded Hilbert $B$-module $\E$, a graded $^\ast$-representation $\rho\colon A \rightarrow \bB_B(\E)$ and an odd unbounded regular self-adjoint operator $\st{D}$ such that (1) $\rho(a)(\st{D}^2+1)^{-1} \in \bK_B(\E)$, and (2) for $a$ in a dense subalgebra of $A$, $\rho(a)$ preserves the domain $\dom(\st{D})$ and $[\st{D},\rho(a)] \in \bB_B(\E)$.
\end{definition}
We describe a prototypical example in the next section.  Given an unbounded cycle $(\E,\rho,\st{D})$, one obtains a bounded cycle $(\E,\rho,F)$ with
\[ F=b(\st{D}), \qquad b(x)=x(1+x^2)^{-1/2}.\]
The K-homology class represented by the triple $(\E,\rho,\st{D})$ is defined to be the class represented by the `bounded transform' $(\E,\rho,F)$.

We will state a sufficient condition for an unbounded cycle to represent a KK product \cite{KucerovskyUnbounded}.  The general condition is somewhat technical-looking, but in our application later it will simplify (in particular, we only use the criterion in the special case in which the $C^\ast$ algebra $C=\bC$, in which case $\E$, $\E_2$ are ordinary Hilbert spaces).  Let $(\E_i,\rho_i,\st{D}_i)$, $i=1,2$ be cycles for the pairs of $C^\ast$ algebras $(A,B)$, $(B,C)$ respectively.  Let
\[ \E=\E_1 \wh{\otimes}_{\rho_2} \E_2, \qquad \rho=\rho_1\wh{\otimes} 1 \colon A \rightarrow \bB_C(\E).\]
For each $e \in \E_1$ define
\[ \T_e \colon e_2 \in \E_2 \mapsto e \wh{\otimes} e_2 \in \E.\]
Its adjoint $\T_e^\ast$ is
\[ \T_e^\ast \colon e_1\wh{\otimes} e_2 \mapsto \rho_2\big((e,e_1)_B\big)e_2 \]
where $(-,-)_B$ denotes the $B$-valued inner product on $\E_1$.
\begin{proposition}[Sufficient condition for products of unbounded cycles, \cite{KucerovskyUnbounded}]
\label{prop:SuffCond}
The unbounded cycle $(\E,\rho,\st{D})$ represents the Kasparov product of $[(\E_1,\rho_1,\st{D}_1)]$, $[(\E_2,\rho_2,\st{D}_2)]$ if the following conditions hold:
\begin{enumerate}
\item \emph{(Connection condition.)} There is a dense subset of $e \in \rho_1(A)\E_1$ such that
\[ \T_e \st{D}_2-(-1)^{\deg(e)}\st{D} \T_e, \qquad \T_e^\ast\st{D}-(-1)^{\deg(e)} \st{D}_2\T_e^\ast \]
extend to (bounded) adjointable operators from $\E_2$ to $\E$ (resp. $\E$ to $\E_2$).
\item \emph{(Semi-boundedness condition.)} The domain $\dom(\st{D}) \subset \dom(\st{D}_1\otimes 1)$, and there is a constant $k$ such that for $e \in \dom(\st{D})$, the following inequality holds between elements of the $C^\ast$ algebra $C$:
\[ \big(\st{D}e,(\st{D}_1\otimes 1)e\big)_C+\big((\st{D}_1\otimes 1)e,\st{D}e\big)_C \ge k(e,e)_C\]
\end{enumerate}  
\end{proposition}

\subsection{Elliptic operators on complete manifolds.}
This section reviews some facts about $1^{st}$ order symmetric elliptic operators.  Such operators yield examples of unbounded K-homology cycles.  Proofs and further details can be found in \cite[Chapter 10]{HigsonRoe}.

Throughout this section, $\st{D}\colon C^\infty_c(\Y,S) \rightarrow C^\infty_c(\Y,S)$ denotes a $1^{st}$ order symmetric differential operator acting on sections of a Hermitian vector bundle over a Riemannian manifold $(\Y,g)$.  The symbol of $\st{D}$ will be denoted
\[ \sigma_{\st{D}} \colon T^\ast \Y \rightarrow \End(S).\] 
\begin{definition}
The \emph{propagation speed} of $\st{D}$ is
\[ c_{\st{D}}=\sup \{\|\sigma_{\st{D}}(y,v)\|:y \in \Y,v \in T_y^\ast \Y,\|v\|=1\}.\]
If $c_{\st{D}}<\infty$ then we say $\st{D}$ has \emph{finite propagation speed}.
\end{definition}

\begin{proposition}[Chernoff \cite{ChernoffSelfAdjoint}]
If $\Y$ is complete and $\st{D}$ has finite propagation speed then $\st{D}$ is essentially self-adjoint.
\end{proposition}
\begin{example}\label{ex:DiracOp}
An important special case is a spin-c Dirac operator.  Let $S \rightarrow \Y$ be a \emph{spinor module} over an even-dimensional Riemannian manifold $(\Y,g)$, i.e. $S$ is a $\bZ_2$-graded Hermitian vector bundle equipped with an isomorphism $\c \colon \Cl(T\Y) \xrightarrow{\sim} \End(S)$, where $\Cl(T\Y)$ denotes the Clifford algebra bundle.  Let $\nabla^S$ be a Hermitian connection on $S$, which is compatible with the Clifford action in the sense that
\[ \nabla^S_X(\c(Y)s)=\c(Y)\nabla_X^S s+\c(\nabla_XY)s, \quad s \in C^\infty(Y,S)\]
where $\nabla$ is the Levi-Civita connection, $X$, $Y$ smooth vector fields.  The Dirac operator $\st{D}\colon C^\infty_c(Y,S) \rightarrow C^\infty_c(Y,S)$ is defined by the composition
\[ C^\infty_c(Y,S) \xrightarrow{\nabla} C_c^\infty(Y,T^\ast \Y \otimes S) \xrightarrow{g^\sharp} C_c^\infty(Y,T\Y \otimes S) \xrightarrow{\c} C^\infty(Y,S).\]
Usually we omit $g^\sharp$ from the notation, the identification $TY \simeq T^\ast Y$ being understood.  Let $e_n$, $n=1,..,\dim(\Y)$ be a local orthonormal frame, then
\[\st{D}=\sum_{n=1}^{\dim(\Y)} \c(e_n)\nabla_{e_n}^S\]
locally.  The symbol of $\st{D}$ is
\[ \sigma_{\st{D}}(\xi)=\i\c(\xi).\]
Since $\sigma_{\st{D}}(\xi)^2=|\xi|^2$, $\st{D}$ is elliptic with $c_{\st{D}}=1$.  If $\Y$ is complete, then $\st{D}$ is essentially self-adjoint.
\end{example}

The following is a well-known consequence of the Rellich lemma (cf. \cite[Proposition 10.5.2]{HigsonRoe}).
\begin{proposition}
\label{RellichLemma}
Let $\st{D}$ be an essentially self-adjoint, $1^{st}$ order elliptic operator.  For any $\chi \in C_0(\bR)$ and $f \in C_0(\Y)$ the operator
\[ f \cdot \chi(\st{D})\colon L^2(\Y,S) \rightarrow L^2(\Y,S) \]
is compact.
\end{proposition}
Using Proposition \ref{RellichLemma}, one shows (cf. \cite[Theorem 10.6.5]{HigsonRoe}):
\begin{proposition}
\label{prop:KHomologyClass}
Let $\st{D}$ be an essentially self-adjoint, odd, $1^{st}$-order elliptic operator acting on sections of a $\bZ_2$-graded Hermitian vector bundle $S$.  Let $\rho$ be the representation of $C_0(Y)$ on $L^2(\Y,S)$ by multiplication operators.  The triple $(L^2(\Y,S),\rho,\st{D})$ is an unbounded cycle representing a class $[\st{D}] \in \K^0(C_0(\Y))$.
\end{proposition}
If $\Y$ is even-dimensional and complete, then by Example \ref{ex:DiracOp}, a Dirac operator $\st{D}$ acting on sections of a spinor module satisfies the conditions of Proposition \ref{prop:KHomologyClass}.

\begin{proposition}[cf. \cite{HigsonRoe} Lemma 10.5.5]
\label{PropagationDistance}
Let $\st{D}$ be an essentially self-adjoint $1^{st}$-order operator with propagation speed $c_{\st{D}}<\infty$.  Suppose $\chi \in C_0(\bR)$ has (distributional) Fourier transform $\wh{\chi}$ supported in $(-r,r)$ for some $r>0$.  Let $f,g$ be continuous bounded functions on $\Y$ such that the Riemannian distance between $\supp(f)$ and $\supp(g)$ is greater than $r\cdot c_{\st{D}}$.  Then
\[ f \chi(\st{D})g=0.\]
\end{proposition}

\subsection{Crossed products and the descent map.}\label{sec:CrossProdDescent}
Let $G$ be a locally compact group with Haar measure $dg$, and let $A$ be a $G$-$C^\ast$ algebra.  One can define a new $C^\ast$ algebra $G \ltimes A=C^\ast(G,A)$ called the \emph{crossed product} algebra.  In this section we briefly describe this algebra, following \cite{Blackadar, KasparovNovikov}.

The space $C_c(G,A)$ of continuous, compactly supported functions $a\colon G\rightarrow A$ can be made into an involutive algebra with the convolution product
\[ (a_1\cdot a_2)(g)=\int_G a_1(g_1)g_1.a_2(g_1^{-1}g)\,\,dg_1, \]
and involution $a^\ast(g)=g.(a(g^{-1}))^\ast \mu(g)^{-1}$, where $\mu$ is the modular homomorphism of $G$.  The completion of this algebra in the maximal $C^\ast$-norm (cf. \cite[Section 10]{Blackadar}) yields the \emph{crossed product} $C^\ast$ \emph{algebra} $G \ltimes A=C^\ast(G,A)$ (we use both notations interchangeably).  

The crossed product has an important universal property: there is a 1-1 correspondence between $^\ast$-representations $\rho_{G\ltimes A}$ of $G\ltimes A$, and `covariant pairs' $(\rho_G,\rho_A)$ consisting of a unitary representation $\rho_G$ of $G$ and a $^\ast$-representation $\rho_A$ of $A$ on a common Hilbert space, which satisfy
\[ \rho_G(g)\rho_A(a)\rho_G(g)^{-1}=\rho_A(g.a).\]
Given $(\rho_G,\rho_A)$, the representation $\rho_{G\ltimes A}$ is sometimes called the `integrated form' of the covariant pair.  The special case $A=\bC$ gives rise to the group $C^\ast$ algebra $G\ltimes \bC=:C^\ast(G)$; in this case the universal property says that there is a 1-1 correspondence between unitary representations of $G$ and $^\ast$-representations of $C^\ast(G)$.

Let $B$ be a $G$-$C^\ast$ algebra, and let $(\E,(\cdot,\cdot)_B)$ be a right Hilbert $B$-module.  The algebra $C_c(G,B)$ acts on $C_c(G,\E)$ on the right according to the formula:
\[ (e\cdot b)(g)=\int_G e(g_1)\cdot g_1.b(g_1^{-1}g)\,\,dg_1.\]
$C_c(G,\E)$ carries a $C_c(G,B)$-valued inner product defined by
\[ (e_1,e_2)_{G\ltimes B}(g)=\int_G g_1^{-1}.(e_1(g_1),e_2(g_1g))_B\,\,dg_1. \]
The completion of $C_c(G,\E)$ in the corresponding norm is denoted $G \ltimes \E=C^\ast(G,\E)$ and carries a $G \ltimes B$-valued inner product.

The above construction leads to a homomorphism
\[ j_G \colon \KK_G(A,B) \rightarrow \KK(G\ltimes A,G\ltimes B) \]
known as the \emph{descent map}, which is functorial with respect to the Kasparov product.  At the level of cycles, it sends $(\E,\rho,F)$ to a cycle $(G\ltimes \E,\ti{\rho},\ti{F})$, where the operator $\ti{F}$ is obtained by simply applying $F$ point-wise, and the representation of $G\ltimes A$ is given by
\[ (\ti{\rho}(a) e)(g)=\int_G \rho(a(g_1))g_1.e(g_1^{-1}g)\,\,dg_1.\]

\subsection{The $K$-index.}
Let $K$ be a compact Lie group equipped with a Haar measure $dk$.  Let $H=H^+\oplus H^-$ be a $\bZ_2$-graded Hilbert space equipped with a continuous action of $K$ by unitary transformations (homogeneous of degree $0$).  The action of $K$ on $H$ induces a $^\ast$-representation $\rho$ of the group $C^\ast$-algebra $C^\ast(K)$:
\[(a\cdot v)(y)=\int_K a(k) k.v\,\,dk.\]

Any $K$-invariant closed subspace $W \subset H$ decomposes into an orthogonal direct sum of $K$-isotypic components
\[ W=\bigoplus_{\pi \in \Irr{K}} W_\pi.\]
Let $\st{D}$ be an odd, self-adjoint, $K$-equivariant operator on $H$ (possibly unbounded).  With respect to the decomposition $H=H^+\oplus H^-$, $\st{D}$ takes the form
\[ \st{D}=\left(\begin{array}{cc} 0 & \st{D}^-\\ \st{D}^+ & 0\end{array}\right), \qquad (\st{D}^+)^\ast=\st{D}^-.\]
\begin{definition}
\label{def:Kindex}
$\st{D}$ is $K$-\emph{Fredholm} if, for each irreducible representation $\pi$ of $K$, the restriction of $\st{D}$ to the $\pi$-isotypic component $H^{\pi}$ is Fredholm.  In this case we define the $K$-index
\[ \index_K(\st{D})=\sum_{\pi \in \Irr(K)} \Big(\dim(\ker(\st{D}^+)_\pi)-\dim(\ker(\st{D}^-)_\pi)\Big)\pi \in R^{-\infty}(K).\]
\end{definition} 

As a consequence of the Peter-Weyl theorem,
\begin{equation} 
\label{GroupAlgebra}
C^\ast(K) \simeq \bigoplus_{\pi \in \Irr(K)} \End(\pi)
\end{equation}
where $\oplus$ here means the completion of the algebraic direct sum in the sup-norm.  For example, when $K=T$ is a torus with integral lattice $\Lambda=\ker(\exp\colon \t \rightarrow T)$, all irreducible representations are 1-dimensional and correspond to weights $\lambda \in \Lambda^\ast=\Hom(\Lambda,\bZ)$, hence $C^\ast(T)\simeq C_0(\Lambda^\ast)$ (the Gelfand/Pontryagin dual).  Using \eqref{eqn:KHomologyAdditivity} and \eqref{GroupAlgebra} one has
\begin{equation} 
\label{eqn:IsoRepRing}
\K^0(C^\ast(K)) \simeq R^{-\infty}(K),
\end{equation}
the formal completion of the representation ring of $K$.   

Examples of $K$-Fredholm operators come from cycles for $\K^0(C^\ast(K))$. Indeed, if $(H,\rho,\st{D})$ is an unbounded cycle for $\K^0(C^\ast(K))$, then $\st{D}$ is $K$-Fredholm, and its $K$-index is the image of $[\st{D}]$ under the isomorphism \eqref{eqn:IsoRepRing}.

There are interesting operators having infinite dimensional kernel or cokernel, but with well-defined $K$-index.  Examples include transversally elliptic operators \cite{AtiyahTransEll} and the deformed Dirac operators on non-compact manifolds treated by Braverman \cite{Braverman2002}.

\section{A Fredholm property for $1^{st}$ order elliptic operators.}\label{sec:FredholmProperty}
In this section we prove a Fredholm property for $1^{st}$ order elliptic operators $\st{D}$ on non-compact manifolds satisfying a certain condition with respect to the action of a semi-direct product group $K\ltimes \Gamma$, with $K$ compact and $\Gamma$ discrete.

\subsection{$(\Gamma,K)$-admissible vector bundles.}\label{sec:GammaK}
Let $\Gamma$ be a countable discrete group, equipped with a proper length function, that is, a function $l \colon \Gamma \rightarrow \bR_{\ge 0}$ satisfying
\[ l(\gamma)=0 \hspace{0.2cm} \Leftrightarrow \hspace{0.2cm} \gamma=1, \qquad l(\gamma_1\gamma_2)\le l(\gamma_1)+l(\gamma_2), \qquad l(\gamma^{-1})=l(\gamma),\]
and such that the image of any infinite subset of $\Gamma$ is unbounded.  For a countable discrete group, length functions always exist (cf. \cite{DoranPark}, p.5).  Let $K$ be a compact Lie group that acts on $\Gamma$ by continuous group automorphisms $k \colon \gamma \mapsto \gamma^k$, and let $K \ltimes \Gamma$ denote the corresponding semi-direct product group.  Since the action of $K$ on $\Gamma$ is continuous, the action descends to an action of the component group $K/K^0$.  In particular, we see that the orbits are finite, with cardinality no greater than the number of components of $K$.  By averaging, we may assume that the length function is $K$-invariant, and thus for a $K$-orbit $\O \subset \Gamma$ we will write $l(\O)$ for the common length of all the elements in $\O$.

Let $(\Y,g)$ be a complete Riemannian manifold carrying an isometric action of $K \ltimes \Gamma$.  Assume the action of $\Gamma$ is free and proper.  Let 
\[ \pi \colon \Y \rightarrow Y:=\Y/\Gamma\] 
be the quotient map.  Since $\pi(k \cdot \gamma \cdot y)=\pi(\gamma^k \cdot k \cdot y)=\pi(k\cdot y)$, the $K$ action descends to $Y$. 

Let $S$ be a $\bZ_2$-graded vector bundle, equipped with actions of $K$ and of some $U(1)$ central extension $\wh{\Gamma}$ of $\Gamma$, covering the actions of $K$, $\Gamma$ on $\Y$.  Assume $S$ has a Hermitian structure invariant under $K$ and $\wh{\Gamma}$, and use this to define the Hilbert space $L^2(\Y,S)$ of square-integrable sections.
\begin{definition}
\label{def:GammaKAdmissible}
Let $S$ be a Hermitian vector bundle over $\Y$ carrying actions of $K$ and $\hGamma$ as described above.  We say $S$ is $(\Gamma,K)$-\emph{admissible} if for any $a \in C^\ast(K)$ and $s \in L^2(\Y,S)$,
\begin{equation}
\label{eqn:AdmissibleCond} 
\|a \cdot \hgamma \cdot s \| \xrightarrow{l(\gamma) \rightarrow \infty} 0.
\end{equation}
Here $\hgamma \in \hGamma$ denotes any lift of $\gamma \in \Gamma$.
\end{definition}
\begin{remark}
If $S$ is $(\Gamma,K)$-admissible, then elements of $K$, $\wh{\Gamma}$ must be far from commuting in some sense, since otherwise one would have 
\[ \|a\cdot \hgamma \cdot s\|=\|\hgamma \cdot a \cdot s\|=\|a\cdot s\|, \]
as the action of $\wh{\Gamma}$ on $L^2(\Y,S)$ is isometric.
\end{remark}

\begin{example}
\label{ex:Disconnected}
Let $K$ be a torus and $\Gamma$ any countable discrete group.  Let $\Lambda^\ast \simeq \Irr(K)$ be the weight lattice of $K$.  Let $\Y=\Gamma$ with trivial $K$ action and $\Gamma$ acting by left translation.  Let $S=\Y \times \bC$, with $\Gamma$ acting only on the first factor, i.e. $\gamma.(y,z)=(\gamma.y,z)$ for $(y,z) \in \Y \times \bC$.  Choose any map $\tn{wt} \colon \Y \rightarrow \Lambda^\ast$.  Let $K$ act on the fibre $S_y$ with weight $\tn{wt}(y) \in \Lambda^\ast$, thus $S$ becomes a $K$-equivariant line bundle over $\Y$, with fibres $\bC_{\tn{wt}(y)}$.  One has $C^\ast(K)\simeq C_0(\Lambda^\ast)$ by Pontryagin duality, and the action of $f \in C_0(\Lambda^\ast)$ on $L^2(\Y,S)$ is by multiplication by the function $f\circ \tn{wt}$.  It follows that $S$ is $(\Gamma,K)$-admissible if and only if the map $\wt$ has finite fibres.  (This is also equivalent to saying that the zero operator on $L^2(\Y,S)$ is $K$-Fredholm, consistent with Theorem \ref{thm:GammaKAdmissible} below.)
\end{example}

\begin{example}
\label{ex:Cylinder}
The simplest connected example is a line bundle over a cylinder $\Y=\bR \times \bR/\bZ$, equipped with actions of $K=S^1$ and $\Gamma=\bZ$.  Let $x \in \bR$, $y \in \bR/\bZ$ be coordinates on $\Y$.  The line bundle $L=\Y \times \bC$ has $S^1=\bR/\bZ$ and $\bZ$ actions given respectively by 
\[ \rho_{\scriptscriptstyle S^1}(\theta)(x,y;z)=(x,y+\theta;z), \qquad \rho_{\scriptscriptstyle \bZ}(n)(x,y;z)=(x+n,y;e^{2\pi \i ny}z).\]
We have the commutation relation
\[ \rho_{\scriptscriptstyle \bZ}(n)\rho_{\scriptscriptstyle S^1}(\theta)\rho_{\scriptscriptstyle \bZ}(n)^{-1}\rho_{\scriptscriptstyle S^1}(\theta)^{-1}=e^{2\pi \i n\theta}.\] 
Thus $L$ carries an action of a central extension of $S^1 \times \bZ$ covering the action of $S^1 \times \bZ$ on the cylinder, and in fact $L$ is $(\bZ,S^1)$-admissible (see Proposition \ref{prop:LambdaTAdmissible}).
\end{example}
The next result gives a sense of the $(\Gamma,K)$-admissible condition.
\begin{proposition}
Let $S \rightarrow \Y$ be a $(\Gamma,K)$-admissible line bundle.  Assume $Y=\Y/\Gamma$ is compact, and that $S$ admits a connection $\nabla$ that is both $K$ and $\hGamma$ invariant.  Let $\mu \colon \Y \rightarrow \k^\ast$ be the moment map defined by Kostant's formula
\[ 2\pi \i \pair{\mu}{\xi}=2\pi \i \mu_\xi=\L_\xi-\nabla_{\xi_{\sY}}.\]
Then $\mu$ is proper.
\end{proposition}
\begin{proof}
Since the moment map only depends on the infinitesimal action, we may as well assume $K$ is connected, hence the action of $K$ on $\Gamma$ is trivial, and the actions of $K$, $\Gamma$ on $\Y$ commute.  For $\pi \in \Irr(K)$ and $s \in L^2(\Y,S)=H$ let $s_\pi$ denote the component of $s$ in the $\pi$-isotypical component $H_\pi$.  Thus
\[ s=\sum_{\pi \in \Irr(K)} s_\pi, \qquad \|s\|^2=\sum_{\pi \in \Irr(K)}\|s_\pi\|^2.\]
As the action of $\hGamma$ is isometric,
\begin{equation} 
\label{eqn:ConsMass}
\|s\|^2=\|\hgamma \cdot s\|^2=\sum_{\pi \in \Irr(K)} \|(\hgamma \cdot s)_{\pi}\|^2.
\end{equation}
Using the description of $C^\ast(K)$ given in \eqref{GroupAlgebra}, let
\[ a=\bigoplus_{\pi \in \Irr(K)} a_\pi \id_\pi, \qquad a_\pi \in \bC, \quad a_\pi \xrightarrow{\pi \rightarrow \infty} 0.\]
Then $(a\cdot s)_\pi=a_\pi s_\pi$.  By the $(\Gamma,K)$-admissible condition
\[ \|a \cdot \hgamma \cdot s\|^2=\sum_{\pi \in \Irr(K)}a_\pi^2\|(\hgamma \cdot s)_\pi\|^2 \xrightarrow{l(\gamma)\rightarrow \infty} 0.\]
In particular, for any fixed $\pi$, the norm $\|(\hgamma \cdot s)_{\pi}\| \rightarrow 0$ as $l(\gamma) \rightarrow \infty$.  

Let $\xi_j$, $j=1,...,\dim(\k)$ be an orthonormal basis of $\k$.  By definition $\L_{\xi_j}s_\pi=\pi(\xi_j)s_\pi$.  Thus
\[ \sum_j \|\L_{\xi_j}s\|^2=\sum_{\pi \in \Irr(K)}\sum_j \big(\pi(\xi_j)^2s_\pi,s_\pi\big)=\sum_{\pi \in \Irr(K)}\tn{Cas}_\pi \|s_\pi\|^2,\]
where $\tn{Cas}_\pi$ is the value of the Casimir operator in the representation $\pi$.  Recall $\tn{Cas}_\pi \rightarrow \infty$ as $\pi \rightarrow \infty$.  Applying this to the section $\hgamma \cdot s$ and using $\|(\hgamma \cdot s)_\pi\| \rightarrow 0$, as well as equation \eqref{eqn:ConsMass} we deduce that
\begin{equation} 
\label{eqn:LieAct}
\sum_j \|\L_{\xi_j}(\hgamma\cdot s)\|^2 \xrightarrow{l(\gamma)\rightarrow \infty} \infty.
\end{equation}
By Kostant's formula
\[ 2\pi\|\mu_{\xi_j} (\hgamma \cdot s)\| \ge \|\L_{\xi_j}(\hgamma \cdot s)\|-\|\nabla_{\xi_j}(\hgamma \cdot s)\|.\]
Since by assumption $\nabla$ is $\hGamma$-invariant, the norm $\|\nabla_{\xi_j}(\hgamma \cdot s)\|$ is independent of $\hgamma$.  Thus
\[ \sum_j 2\pi\|\mu_{\xi_j} (\hgamma \cdot s)\| \ge \sum_j\|\L_{\xi_j}(\hgamma \cdot s)\| - C_s \]
where $C_s$ is a constant not depending on $\hgamma$.  By \eqref{eqn:LieAct},
\begin{equation} 
\label{eqn:AvgNormMu}
\sum_j \|\mu_{\xi_j}(\hgamma \cdot s)\| \xrightarrow{l(\gamma)\rightarrow \infty} \infty.
\end{equation}
If, for example, we choose $s$ with support contained in a fundamental domain $\Delta$ for the $\Gamma$ action, \eqref{eqn:AvgNormMu} implies that the supremum of $\sum_j |\mu_{\xi_j}|$ over $\gamma \cdot \Delta$ goes to infinity as $l(\gamma)$ goes to infinity.  Since $\nabla$ is $K$ and $\hGamma$ equivariant, $\mu$ satisfies $\iota(\xi_{\Y})\tn{curv}_\nabla=-d\mu_\xi$ where $\tn{curv}_\nabla \in \Omega^2(\Y)^{K\times \Gamma}$ is the curvature 2-form.  Since the action of $\Gamma$ is cocompact, this implies the gradient (for any $\Gamma$-invariant metric on $\Y$) of $\mu_\xi$ is uniformly bounded.  Combined with \eqref{eqn:AvgNormMu}, it follows that the infimum of $\sum_j |\mu_{\xi_j}|$ over $\gamma \cdot \Delta$ goes to infinity as $l(\gamma)$ goes to infinity.  By cocompactness of the $\Gamma$ action, $\sum_j |\mu_{\xi_j}|$ is a proper map. 
\end{proof}
\begin{proposition}
\label{prop:admissiblecompact}
Let $S \rightarrow \Y$ be $(\Gamma,K)$-admissible.  Let $T$ be a compact operator on $L^2(\Y,S)$.  Then for any $a \in C^\ast(K)$
\[ \lim_{l(\O)\rightarrow \infty} \sum_{\gamma \in \O} \|a\cdot \hgamma \cdot T \.\hgamma^{-1} \|=0,\]
where $\O \subset \Gamma$ denotes a $K$-orbit, and $\hgamma \in \hGamma$ is any lift of $\gamma \in \Gamma$.
\end{proposition}
\begin{proof}
We first show the analogous result for a single summand, i.e. $\|a\.\hgamma \. T\. \hgamma^{-1}\| \rightarrow 0$ as $l(\gamma) \rightarrow \infty$.  For $T$ of rank $1$, this is essentially a re-phrasing of the definition.   By taking sums one obtains the result for all finite rank $T$.  For a general compact $T$, fix $\epsilon >0$ and choose a finite rank $T^\prime$ such that $\|T-T^\prime\| \le (2\|a\|)^{-1}\epsilon$.  Choose $n$ sufficiently large that for $l(\gamma)>n$, $\|a\.\wh{\gamma}\.T^\prime \. \wh{\gamma}^{-1}\| \le \epsilon/2$.  Then for $l(\gamma)>n$,
\[ \|a\.\wh{\gamma}\.T\. \wh{\gamma}^{-1}\| \le \|a\|\.\|T-T^\prime\|+\|a\.\wh{\gamma}\.T^\prime\.\wh{\gamma}^{-1}\| \le \epsilon.\]

To handle the sum over $\O$, recall the cardinality of $\O$ is uniformly bounded by the number of components $c(K)$ of $K$.  Let $\epsilon >0$.  Using what we have already proved, let $n$ be such that $\|a\.\hgamma \.T \.\hgamma^{-1}\|<\tfrac{\epsilon}{c(K)}$ for all $\gamma$ such that $l(\gamma)>n$.  Then
\[ l(\O)>n \quad \Rightarrow \quad \sum_{\gamma \in \O}\|a\.\hgamma\. T \.\hgamma^{-1}\|<c(K)\cdot \tfrac{\epsilon}{c(K)}=\epsilon.\]
\end{proof}
The main result of this section is the following.
\begin{theorem}
\label{thm:GammaKAdmissible}
Let $\Y$ be a $K\ltimes \Gamma$-manifold, equipped with a complete $K\ltimes \Gamma$-invariant Riemannian metric.  Assume $\Gamma$ acts freely and properly on $\Y$, and let $Y=\Y/\Gamma$.  Let $S \rightarrow \Y$ be $(\Gamma,K)$-admissible.  Let $\st{D}$ be an odd, essentially self-adjoint $1^{st}$-order elliptic operator with finite propagation speed, equivariant for the actions of $K$ and $\wh{\Gamma}$, acting on sections of $S$.  Then $\st{D}$ defines a class $[\st{D}] \in \K^0(K \ltimes C_0(Y))$.
\end{theorem}
In particular, if $Y=\Y/\Gamma$ is compact then $\st{D}$ is $K$-Fredholm with $K$-index
\[ \index_K(\st{D})=p_\ast [\st{D}] \in \K^0(C^\ast(K))=R^{-\infty}(K),\]
where $p\colon Y \rightarrow \pt$.  More generally, the intersection product in $\KK$-theory gives a map
\begin{equation} 
\label{CapProduct}
\otimes \colon \KK(C^\ast(K),K \ltimes C_0(Y)) \times \KK(K\ltimes C_0(Y),\bC) \rightarrow \KK(C^\ast(K),\bC).
\end{equation}
Using the descent homomorphism, we obtain a pairing
\begin{equation}
\label{eqn:CapProduct2} 
\cap \colon \K^0_K(Y) \times \K^0(K\ltimes C_0(Y)) \rightarrow \K^0(C^\ast(K)), \qquad (\st{x},[\st{D}]) \mapsto \st{x} \cap [\st{D}]=j_K(\st{x})\otimes [\st{D}].
\end{equation}
Thus as input $\cap$ takes an element in the $K$-equivariant $\K$-theory of $Y$ and an element of $\K^0(K\ltimes C_0(Y))$ (where our K-homology class of interest lives) and outputs an `index' in $R^{-\infty}(K)$.  
\begin{definition}
\label{def:CapProd}
We will refer to the map defined in equation \eqref{eqn:CapProduct2} as the \emph{cap product} or \emph{index pairing}, as this instance of the $\KK$-product is a generalization of the usual pairing between K-theory and K-homology.
\end{definition}

\begin{example}
To illustrate Theorem \ref{thm:GammaKAdmissible} we return to Example \ref{ex:Cylinder}.  Identify $\Y=\bR^2/\bZ$ with $\bC/\bZ$, and let $w=x+\i y$.  For the $\bZ_2$-graded vector bundle $S$ we take $\wedge \bC \otimes L=L \oplus (L\otimes d\ol{w})$.  The operator $\st{D}$ will be a Dolbeault-Dirac operator on $\wedge \bC$ twisted by $L$.  We must choose a connection on $L$ invariant under the $S^1 \ltimes \wh{\bZ}$ action.  We may take the connection 1-form on the corresponding principal $\bC^{\times}$-bundle to be
\[ \frac{dz}{z} - 2\pi \i x dy \]
since this is invariant under $(x,y;z)\mapsto (x+n,y;e^{2\pi \i ny}z)$ for $n \in \bZ$.  The corresponding connection 1-form on the base is $A=-2\pi \i xdy$.  The Dirac operator is
\[ \st{D}=\sqrt{2}\big(\ol{\partial}+\ol{\partial}^\ast\big) + \sqrt{2} \pi x\big(\varepsilon_{d\ol{w}}+\iota_{dw}\big), \qquad \ol{\partial}=\varepsilon_{d\ol{w}} \partial_{\ol{w}}, \quad \ol{\partial}^\ast=-\iota_{dw}\partial_w\]
where $\varepsilon_{d\ol{w}}$ denotes exterior multiplication and $\iota_{dw}$ is contraction ($\iota_{dw}d\ol{w}=2$).  The kernel of the operator $\st{D}^+$ mapping sections of $L \otimes \wedge^0\bC$ to sections of $L \otimes \wedge^1\bC$ consists of $L^2$ solutions of
\begin{equation} 
\label{eqn:ODE}
\partial_{\ol{w}} f=-\pi x f
\end{equation}
which are periodic in $y \in \bR/\bZ$.  A basis for the space of solutions is the family
\[ f_n(w,\ol{w})=e^{2\pi nw} e^{-\pi x^2}=e^{2\pi \i ny}e^{2\pi nx-\pi x^2}, \quad n \in \bZ. \]
Under the $S^1$ action $f_n$ transforms with weight $n$.  Hence
\begin{equation} 
\label{eqn:KIndexExample}
\ker(\st{D}^+) \simeq \bigoplus_{n \in \bZ} \bC_n \simeq L^2(S^1). 
\end{equation}
The $L^2$ kernel of $\st{D}^-$ is trivial (similar to \eqref{eqn:ODE} but with the sign reversed).  Thus the $K=S^1$-index in this case is \eqref{eqn:KIndexExample}, which is what one expects for the quantization of the cotangent bundle of $S^1$.
\end{example}

In the case that the actions of $K$ and $\wh{\Gamma}$ on $S$ fit together into an action of a $U(1)$ central extension $G=K\ltimes \wh{\Gamma}$ of $K\ltimes \Gamma$, there is a small refinement of Theorem \ref{thm:GammaKAdmissible} that we will use later on.  

The subgroup $\wh{K}=K\times U(1)$ is a union of connected components of $G$.  As the inclusion map 
\[ \iota_{\wh{K}} \colon \wh{K} \hookrightarrow G \] 
is a group homomorphism with open range, it induces an injective $^\ast$-homomorphism $C^\ast(\wh{K}) \rightarrow C^\ast(G)$ also denoted $\iota_{\wh{K}}$ (cf. \cite[Section 7]{ChrisPhillipsLecNotes}).  By Fourier decomposition over $S^1$, the $C^\ast$ algebras $C^\ast(G)$, $C^\ast(\wh{K})$ break up into infinite direct sums (over the Pontryagin dual $\bZ$ of $S^1$) of their $S^1$-homogeneous ideals:
\begin{equation} 
\label{eqn:HomogeneousIdeals}
C^\ast(G) = \bigoplus_{n \in \bZ} C^\ast(G)^{(n)}, \qquad C^\ast(\wh{K})=\bigoplus_{n \in \bZ} C^\ast(\wh{K})^{(n)}
\end{equation}
where $C^\ast(G)^{(n)}$ denotes the closed ideal obtained by taking limits of continuous, compactly supported functions $f \colon G \rightarrow \bC$ satisfying $f(z^{-1}g)=z^nf(g)$ for all $g \in G$, $z \in U(1) \subset G$ (likewise for $C^\ast(\wh{K})^{(n)}$).  The $n=1$ summand\footnote{Or indeed any of the summands.} in \eqref{eqn:HomogeneousIdeals} for $C^\ast(\wh{K})$ may be identified with $C^\ast(K)$; put in other words, there is a 1-1 correspondence between unitary representations of $K$ and unitary representations of $K \times U(1)$ such that $U(1)$ acts with weight $1$.  Thus we may define a $^\ast$-homomorphism
\[ \iota_K \colon C^\ast(K) \rightarrow C^\ast(G) \]
as the composition of the identification $C^\ast(K)\simeq C^\ast(\wh{K})^{(1)}$ with the map $\iota_{\wh{K}}$.  The map in the opposite direction on representations of $G$ with central weight $1$ simply restricts the representation to $K$ (viewed as a subgroup of $G$). More generally, if $A$ is a $G$-$C^\ast$ algebra, then one has an injective $\ast$-homomorphism
\[ \iota_K \colon K \ltimes A \rightarrow G \ltimes A.\]

The same argument we use to prove Theorem \ref{thm:GammaKAdmissible} will prove the following.
\begin{corollary}
\label{cor:Refine}
Let $\Y$ be a $K\ltimes \Gamma$-manifold, equipped with a complete $K\ltimes \Gamma$-invariant Riemannian metric.  Let $S \rightarrow \Y$ be a $G=K\ltimes \wh{\Gamma}$-equivariant vector bundle which is $(\Gamma,K)$-admissible.  Let $\st{D}$ be a $G$-equivariant odd essentially self-adjoint $1^{st}$-order elliptic operator with finite propagation speed acting on sections of $S$.  Then $\st{D}$ defines a class $[\st{D}] \in \K^0(G \ltimes C_0(Y))$.
\end{corollary}
\begin{remark}
\label{rem:Refine}
The restriction $\iota_K^\ast[\st{D}] \in \K^0(K\ltimes C_0(Y))$ recovers the class defined in \ref{def:CapProd}.  Thus the class in Corollary \ref{cor:Refine} is a refinement which remembers the additional symmetry of $\st{D}$.
\end{remark}

The remainder of this section is devoted to the proof of Theorem \ref{thm:GammaKAdmissible}.

\subsection{Partition of unity on $\Y$.}\label{sec:PartUnity}
Let $\{V_j|j \in \J\}$ be a locally finite covering of $Y=\Y/\Gamma$ by $K$-invariant relatively compact open subsets.  The $V_j$ can be chosen sufficiently small that
\begin{enumerate}
\item There is a continuous section $r_j\colon V_j \rightarrow \Y$ of the quotient map $\pi$ over $V_j$.  Define 
\[ U_j=r_j(V_j) \simeq V_j.\]
\item For all $\gamma \ne \gamma^\prime$ one has
\[ \overline{\gamma U_j} \cap \overline{\gamma^\prime U_j}=\emptyset.\]
\end{enumerate}

Let $\{\nu_j \}$, $\{\sigma_j \}$ be collections of $K$-invariant smooth bump functions satisfying
\[ \supp(\nu_j) \subset \supp(\sigma_j) \subset V_j, \qquad \sigma_j|_{\supp(\nu_j)} \equiv 1, \qquad \sum_j \nu_j^2=1. \] 
Lift $\nu_j$, $\sigma_j$ to $U_j$ using the section $r_j$; we also denote the lifts by $\nu_j$, $\sigma_j$ respectively.  For each $\gamma \in \Gamma$ define
\[ \nu_j^\gamma(y)=\nu_j(\gamma^{-1}y), \quad \sigma_j^\gamma(y)=\sigma_j(\gamma^{-1}y), \qquad \forall y \in \Y.\]
Thus $\supp(\nu_j^\gamma) \subset \supp(\sigma_j^\gamma) \subset \gamma \cdot U_j$.  Let
\[ \ti{\nu}_j=\sum_{\gamma \in \Gamma} \nu_j^{\gamma}=\pi^\ast \nu_j, \qquad \ti{\sigma}_j=\sum_{\gamma \in \Gamma} \sigma_j^{\gamma}=\pi^\ast \sigma_j.\]
By construction $\{\ti{\nu}_j^2|j \in \J \}$ is a partition of unity on $\Y$.

\subsection{Mutually orthogonal operators.}
Let $\{F_k|k \in \bN\}$ be a collection of bounded linear operators on a Hilbert space $H$.  It is convenient to make the following definition.
\begin{definition}
The operators $\{F_k\}$ are \emph{mutually orthogonal} if $F_jF_k^\ast=F_k^\ast F_j=0$ for all $j \ne k$.
\end{definition}
\begin{lemma}
\label{lem:orthogonal}
Let $\{F_k\}$ be mutually orthogonal and such that $\sup_k \{\|F_k\|\}=c<\infty$.  Then $\sum_k F_k$ converges in the strong operator topology to an operator of norm $c$.
\end{lemma}
\begin{proof}
Let $F^{(n)}$ be the $n^{th}$ partial sum.  Let $v \in H$ and $v_k$ the orthogonal projection of $v$ onto $\ker(F_k)^\perp=\overline{\ran(F_k^\ast)}$.  For $j \ne k$, $F_jF_k^\ast=0$ implies $v_k \in \ker(F_j)$, and as $v_j \in \ran(F_j) \subset \ker(F_j)^\perp$, we get $\pair{v_j}{v_k}=0$.  Similarly $F_j^\ast F_k=0$ for $j \ne k$ implies $\pair{F_jv_j}{F_kv_k}=0$.  Thus for $m\le n$,
\[ \|(F^{(n)}-F^{(m)})v\|^2=\bigg\|\sum_{k=m}^n F_kv_k\bigg\|^2=\sum_{k=m}^n \|F_kv_k\|^2\le c^2\sum_{k=m}^n\|v_k\|^2,\]
This inequality implies the sum converges in the strong operator topology to an operator $F$ with norm at most $c$.  Putting $v=v_k$ and using $F_jv_k=0$ for $j \ne k$, gives $\|F\| \ge c$.
\end{proof}
\begin{remark}
\label{rem:mutuallyorthogonal}
Let $F \in \bB(L^2(\Y,S))$ be any bounded operator.  Then for any fixed $j \in \J$ and $\gamma^\prime \in \Gamma$ the operators
\[ \sigma_j^{\gamma \gamma^\prime} F \nu_j^{\gamma}, \quad \gamma \in \Gamma\]
are mutually orthogonal, because $\nu_j^{\gamma} \nu_j^{\gamma^\prime}=0=\sigma_j^{\gamma}\sigma_j^{\gamma^\prime}$ for $\gamma \ne \gamma^\prime$.
\end{remark}

\subsection{Proof of Theorem \ref{thm:GammaKAdmissible}.}
We first prove two lemmas.
\begin{lemma}
\label{lem:finitesum}
Fix $j \in \J$ and let $\chi \in C_0(\bR)$ have Fourier transform supported in $(-r,r)$ for some $r>0$.  There exists a finite $K$-invariant subset $\Gamma^\prime \subset \Gamma$ such that
\[ \ti{\sigma}_j \chi(\st{D}) \ti{\nu}_j=\sum_{\gamma \in \Gamma}\sum_{\gamma^\prime \in \Gamma^\prime} \sigma_j^{\gamma \gamma^\prime} \chi(\st{D})\nu_j^{\gamma}.\]
\end{lemma}
\begin{proof}
Since $\supp(\nu_j)$, $\supp(\sigma_j)$ are compact and the action of $\Gamma$ on $\Y$ is proper, there exists a finite subset $\Gamma^\prime \subset \Gamma$ such that for all $\zeta \in \Gamma \setminus \Gamma^\prime$, the distance between $\supp(\sigma_j^\zeta)$ and $\supp(\nu_j)$ is greater than $rc_{\st{D}}$.  Enlarge $\Gamma^\prime$ if necessary to make it $K$-invariant (the orbits of $K$ in $\Gamma$ are finite). Translating by $\gamma \in \Gamma$ and applying Proposition \ref{PropagationDistance} it follows that
\[ \sigma_j^{\gamma \zeta} \chi(\st{D})\nu_j^\gamma=0 \]
for all $\zeta \in \Gamma \setminus \Gamma^\prime$.  The result follows since
\[ \ti{\sigma}_j \chi(\st{D})\ti{\nu}_j=\sum_{\gamma, \gamma^\prime \in \Gamma} \sigma_j^{\gamma \gamma^\prime}\chi(\st{D})\nu_j^\gamma=\sum_{\gamma \in \Gamma}\sum_{\gamma^\prime \in \Gamma^\prime}\sigma_j^{\gamma \gamma^\prime}\chi(\st{D})\nu_j^\gamma.\]
\end{proof}

\begin{lemma}
\label{prop:compactness}
Let $\st{D}$ be an operator as in Theorem \ref{thm:GammaKAdmissible}.  Then for any $a \in C^\ast(K)$, $j \in \J$, and $\chi \in C_0(\bR)$ the operator
\begin{equation} 
\label{eqn:CpctOp}
a\. \ti{\sigma}_j \.\chi(\st{D})\.\ti{\nu}_j 
\end{equation}
is compact.
\end{lemma}
\begin{proof}
By a density argument, it suffices to prove the lemma for $\chi$ with compactly supported Fourier transform.  By Lemma \ref{lem:finitesum}, there is a finite, $K$-invariant subset $\Gamma^\prime \subset \Gamma$ such that \eqref{eqn:CpctOp} may be written
\begin{equation} 
\label{compactnesseqn1}
\sum_{\O \in \Gamma/K} a\. T_{\O}
\end{equation}
where the sum over $\O \in \Gamma/K$ denotes a sum over $K$-orbits in $\Gamma$, and
\[ T_{\O}=\sum_{\gamma \in \O} \hgamma\.T\.\hgamma^{-1}, \qquad T=\sum_{\gamma^\prime \in \Gamma^\prime} \sigma_j^{\gamma^\prime}\chi(\st{D})\nu_j.\]
By Proposition \ref{RellichLemma}, $\sigma_j^{\gamma^\prime}\chi(\st{D})\nu_j$ is compact.  Since $\Gamma^\prime$, $\O$ are finite, $T_{\O}$ is compact, being a finite sum of compact operators.  Therefore it suffices to show that \eqref{compactnesseqn1} converges in norm.

The operators $T_{\O}$ are mutually orthogonal by Remark \ref{rem:mutuallyorthogonal}.  Since $\Gamma^\prime$ is $K$-invariant, $T$ is $K$-invariant.  Since $T_{\O}$ is a sum over a $K$-orbit of translates of $T$, $T_{\O}$ is also $K$-invariant.  Hence the operators $a\.T_{\O}$ are also mutually orthogonal.  Thus
\[ \Bigg\|\sum_{l(\O)>n}a \. T_\O\Bigg\|=\sup_{l(\O)>n}\|a\.T_\O\|\le \sup_{l(\O)>n}\sum_{\gamma \in \O}\|a\.\wh{\gamma}\.T\.\wh{\gamma}^{-1}\|,\]
and the last expression goes to $0$ as $n \rightarrow \infty$ by the $(\Gamma,K)$-admissible condition, Proposition \ref{prop:admissiblecompact}. This proves that the sum in \eqref{compactnesseqn1} converges in norm.
\end{proof}
 
\begin{proof}[\textbf{Proof of Theorem} \ref{thm:GammaKAdmissible}.]
Using Lemma \ref{prop:compactness}, the proof reduces to a computation with commutators.  By Definition \ref{def:Unbounded}, we must check that for a dense set of $f \in K \ltimes C_0(Y)$, $[\st{D},f]$ is bounded and $(1+\st{D}^2)^{-1}\cdot f$ is compact.  We may assume $f$ is of the form $a\otimes \psi$, with $a \in C^\ast(K)$ and $\psi \in C_c^\infty(Y)$.  Since $\st{D}$ is $K$-equivariant, $a$ commutes with $\st{D}$.  The condition on the commutator is immediate, since
\[ [\st{D},\psi]=-\i \sigma_{\st{D}}(\pi^\ast d\psi),\]
$d\psi$ is smooth and compactly supported, hence $\sigma_{\st{D}}(\pi^\ast d\psi)$ is a smooth, bounded endomorphism of $S$.  The element $a \in C^\ast(K)$ also commutes with $\ti{\nu}_j$, $\ti{\sigma}_j$ as these functions are $K$-invariant (even $K\ltimes \Gamma$-invariant).  We have
\begin{align}
\label{eqn:expansion}
\nonumber (1+\st{D}^2)^{-1}&=\sum_j (1+\st{D}^2)^{-1}\ti{\nu}_j^2\\
\nonumber &=\sum_j \ti{\nu}_j (1+\st{D}^2)^{-1}\ti{\nu}_j-[\ti{\nu}_j,(1+\st{D}^2)^{-1}]\ti{\nu}_j\\
\nonumber &=\sum_j \ti{\nu}_j(1+\st{D}^2)^{-1}\ti{\nu}_j+(1+\st{D}^2)^{-1}[\ti{\nu}_j,\st{D}^2](1+\st{D}^2)^{-1}\ti{\nu}_j\\
\nonumber &=\sum_j \ti{\nu}_j(1+\st{D}^2)^{-1}\ti{\nu}_j+(1+\st{D}^2)^{-1}\i\Big(\sigma_{\st{D}}(d\ti{\nu}_j)\st{D}+\st{D}\sigma_{\st{D}}(d\ti{\nu}_j)\Big)(1+\st{D}^2)^{-1}\ti{\nu}_j.
\end{align}
Since $\sigma_j\equiv 1$ on $\supp(\nu_j)$, one has $\sigma_j\nu_j=\nu_j$ and $\sigma_jd\nu_j=d\nu_j$, hence $\sigma_{\st{D}}(d\ti{\nu}_j)=\sigma_{\st{D}}(d\ti{\nu}_j)\ti{\sigma}_j$.  Using this and Lemma \ref{prop:compactness}, one sees that the operator $a\ti{\nu}_j(1+\st{D}^2)^{-1}\ti{\nu}_j$ as well as the operators
\[ a(1+\st{D}^2)^{-1}\sigma_{\st{D}}(d\ti{\nu}_j)\st{D}(1+\st{D}^2)^{-1}\ti{\nu}_j=(1+\st{D}^2)^{-1}\sigma_{\st{D}}(d\ti{\nu}_j)\Big(a\ti{\sigma}_j \st{D}(1+\st{D}^2)^{-1}\ti{\nu}_j\Big),\]
and
\[ a(1+\st{D}^2)^{-1}\st{D}\sigma_{\st{D}}(d\ti{\nu}_j)(1+\st{D}^2)^{-1}\ti{\nu}_j=(1+\st{D}^2)^{-1}\st{D}\sigma_{\st{D}}(d\ti{\nu}_j)\Big(a\ti{\sigma}_j (1+\st{D}^2)^{-1}\ti{\nu}_j\Big) \]
are compact.  Since $\psi$ has compact support in $Y$, $\nu_j \psi=0$ for all but finitely many $j \in \J$.  Thus $a(1+\st{D}^2)^{-1}\psi$ becomes a finite sum of compact operators, hence is compact.
\end{proof}

\section{Application to Hamiltonian loop group spaces}\label{sec:GlobTrans}
This section reviews some results from \cite{LMSspinor}.  Then, using the main result of Section \ref{sec:FredholmProperty}, we construct K-homology classes associated to a Hamiltonian loop group space.  We show that certain index pairings have an additional anti-symmetry under the action of the affine Weyl group, and make a connection with representations of loop groups.  Finally, we determine cycles representing these K-homology classes.

\subsection{Hamiltonian loop group spaces.}\label{sec:LGSpace}
Let $H$ be a compact and connected Lie group with Lie algebra $\h$.  Fix a choice of maximal torus $T$ with Lie algebra $\t$, and let $W=N_H(T)/T$ be the Weyl group.  Let $\Lambda=\ker(\exp \colon \t \rightarrow T)$ be the integral lattice, and $\Lambda^\ast=\Hom(\Lambda,\bZ)$ the (real) weight lattice.  Let $\R \subset \Lambda^\ast$ denote the roots, and fix a positive Weyl chamber $\t_+$ and corresponding set of positive roots $\R_+$.  We have a triangular decomposition
\[ \h_{\bC}=\n_- \oplus \t_{\bC} \oplus \n_+,\]
where $\n_+$ (resp. $\n_-$) is the sum of the positive (resp. negative) root spaces.  The half-sum of the positive roots is
\[ \rho=\frac{1}{2}\sum_{\alpha \in \R_+} \alpha.\]
The affine Weyl group is $W_{\aff}=W\ltimes \Lambda$.  For $w \in W_{\aff}$, let $l(w)$ denote the length of $w$, i.e. the number of simple reflections in a reduced expression for $w$.  Fix an invariant inner product on $\h$, which we use to identify $\h$ with $\h^\ast$, $\t$ with $\t^\ast$ and $\h/\t$ with $\t^\perp$.

Let $LH$ denote the loops $S^1 \rightarrow H$ of some fixed Sobolev level $s > \tfrac{1}{2}$; point-wise multiplication of loops makes $LH$ into a Banach Lie group.  The Lie algebra of $LH$ is the space $L\h=\Omega^0(S^1,\h)$ consisting of loops in $\h$ of Sobolev level $s$.  We define the smooth dual $L\h^\ast$ to consist of $\h$-valued 1-forms on $S^1$ of Sobolev level $s-1$; the pairing between $L\h$, $L\h^\ast$ is given by the inner product, followed by integration over the circle.  $L\h^\ast$ is regarded as the space of connections on the trivial principal $H$-bundle over $S^1$, and carries a smooth, proper $LH$ action by gauge transformations:
\begin{equation}
\label{eqn:GaugeAction} 
h\cdot \xi=\Ad_h \xi - dh h^{-1}, \qquad h \in LH, \quad \xi \in L\h^\ast.
\end{equation}

The group $H$ (hence also any subgroup of $H$, such as $N_H(T)$) embeds in $LH$ as the subgroup of constant loops.  Another important closed subgroup of $LH$ is the \emph{based loop group}
\[ \Omega H=\{h \in LH|h(1)=1\}, \]
and in fact $LH \simeq H \ltimes \Omega H$.  The integral lattice $\Lambda$ may be viewed as a subgroup of $\Omega H$, by identifying $\lambda \in \Lambda$ with the closed geodesic $t\mapsto \exp(t\lambda)$.  The subgroups $T \times \Lambda$ and $N_H(T) \ltimes \Lambda$ are the main examples of groups of the form $K\ltimes \Gamma$ (Section \ref{sec:GammaK}) that will concern us.

\begin{definition}
A \emph{proper Hamiltonian $LH$-space} $(\M,\omega_{\M},\Phi_\M)$ is a Banach manifold $\M$ equipped with a smooth proper action of $LH$, a weakly non-degenerate, $LH$-invariant closed 2-form $\omega$, and a smooth, proper, $LH$-equivariant map
\[ \Phi_{\M} \colon \M \rightarrow L\h^\ast \]
satisfying the moment map condition
\[ \iota(\xi_{\M})\omega_{\M}=-d\pair{\Phi_\M}{\xi}, \qquad \xi \in L\h.\]
\end{definition}
For a more detailed discussion of Hamiltonian loop group spaces, see for example \cite{MWVerlindeFactorization, AlekseevMalkinMeinrenken,BottTolmanWeitsman}.

\subsection{The global transversal of a Hamiltonian loop group space.}\label{sec:GlobTrans}
Let $\Phi_{\M} \colon \M \rightarrow L\h^\ast$ be a proper Hamiltonian $LH$-space.  The based loop group $\Omega H$ acts freely on $L\h^\ast$, and hence also on $\M$.  The quotient
\[ M=\M/\Omega H \]
is a compact finite-dimensional $H$-manifold, and is an example of a \emph{quasi-Hamiltonian $H$-space} \cite{AlekseevMalkinMeinrenken}.  Since $L\h^\ast/\Omega H \simeq H$, $M$ comes equipped with a \emph{group-valued moment map}
\[ \Phi \colon M \rightarrow H.\]

Let $\st{B}_q(\h/\t)$ denote the ball of radius $q>0$ centred at the origin in $\h/\t$.  The normalizer $N_H(T)$ acts on $\st{B}_q(\h/\t)$ by the adjoint action.  Using the inner product there is an $N_H(T)$-equivariant identification $\h/\t \simeq \t^\perp$, where $\t^\perp$ is the orthogonal complement of $\t$ in $\h$.  There is an $N_H(T)$-equivariant map
\[ r_T \colon T \times \st{B}_q(\h/\t)=T \times \st{B}_q(\t^\perp) \rightarrow H, \qquad (t,\xi) \mapsto t\exp(\xi) \]
and for $q$ sufficiently small it is a diffeomorphism onto a tubular neighborhood $U$ of $T$ in $H$.  Define the $N_H(T)$-invariant open submanifold $Y$ of $M$ to be the pre-image:
\[ Y=\Phi^{-1}(U).\]
Let $\Y$ be the $\Lambda$-covering space of $Y$ defined as the fibre product $Y \times_U (\t \times \st{B}_q(\h/\t))$, using the map
\[ r_T \circ (\exp_T,\id) \colon \t \times \st{B}_q(\h/\t) \rightarrow U.\]  
Thus $Y=\Y/\Lambda$ and we have a pullback diagram  
\begin{equation} 
\label{eqn:Pullback}
\xymatrixcolsep{7pc}
\xymatrix{
\Y \ar[r]^{\Phi_{\Y}=(\phi,\phi^{\h/\t})} \ar[d]^{\pi} & \t \times \st{B}_q(\h/\t) \ar[d]^{r_T\circ (\exp_T,\id)} \\
Y \ar[r]_{\Phi|_Y} & U
}
\end{equation}
The first component $\phi$ of the map $\Phi_{\Y}$ defined by \eqref{eqn:Pullback} is a moment map for the $N_H(T)\ltimes \Lambda$-action (using $\t \simeq \t^\ast$), and $\Y$ can be seen to be a degenerate Hamiltonian $N_H(T) \ltimes \Lambda$-space.

Interestingly, $\Y$ can be embedded $N_H(T)\ltimes \Lambda$-equivariantly into the infinite dimensional manifold $\M$, as a small `thickening' of the (possibly) singular closed subset
\[ \X=\Phi_{\M}^{-1}(\t) \]
where $\t \hookrightarrow L\h^\ast$ is embedded as constant connections.  In earlier work \cite[Section 6.4]{LMSspinor} with E. Meinrenken, we showed how to construct such an embedding, depending on the choice of a connection on the principal $\Omega H$-bundle $L\h^\ast \rightarrow H$.\footnote{In \cite{LMSspinor} we actually worked with a slightly larger space $\P H$ (a principal $H$-bundle over $L\h^\ast$), which was desirable for certain purposes, although we have avoided it here for simplicity.  The embedding referred to here can be constructed by the same method.}  In this realization, $\Y$ intersects all the $LH$-orbits in $\M$ transversally, and so we refer to it as a \emph{global transversal} of $\M$.  One reason this perspective is useful is for the description of a certain canonical spin-c structure on $\Y$, explained in the same article \cite{LMSspinor} (see also Remark \ref{rem:Canonical} below).

\subsection{$(\Lambda,T)$-admissible vector bundles on $\Y$.}\label{sec:SpinModY}
Let $G$ be a $U(1)$ central extension of $T \times \Lambda$.  For example, one might obtain such a central extension by pulling back a central extension of the loop group $LH$.  The length function on $\Lambda$ will be given by the norm defined by the inner product on $\t$.  We assume\footnote{Compare to \cite{FHTIII} for example, where this condition is called `topological regularity'.} that there is an injective homomorphism 
\[ \kappa \colon \Lambda \rightarrow \Lambda^\ast \]
such that lifts $\wh{t}$, $\wh{\lambda}$ of $t \in T$, $\lambda \in \Lambda$ respectively, obey the following commutation relation in $G$:
\begin{equation}
\label{eqn:CommutationRelation}
\wh{\lambda}\,\, \wh{t}\,\, \wh{\lambda}^{-1}\,\, \wh{t}^{-1}=t^{-\kappa(\lambda)}.
\end{equation}
Any $U(1)$ central extension of a torus is trivial, hence choosing a trivialization, $G$ is of the form $T \ltimes \hLambda$, for some central extension $\hLambda$ of $\Lambda$.

\begin{proposition}
\label{prop:LambdaTAdmissible}
Let $S \rightarrow \Y$ be a $G$-equivariant Hermitian vector bundle, where $G$ satisfies \eqref{eqn:CommutationRelation}.  Then $S$ is $(\Lambda,T)$-admissible.
\end{proposition}
\begin{proof}
By a density argument, it suffices to show that for each $s \in C^\infty_c(\Y,S)$ and $a \in C^\infty(T)$ we have
\[ \lim_{|\lambda|\rightarrow \infty} \|a\cdot \wh{\lambda} \cdot s\|=0.\]
Using \eqref{eqn:CommutationRelation} we find
\begin{align*} 
a \cdot \wh{\lambda} \cdot s&=\int_T a(t) t\cdot \wh{\lambda} \cdot s \, dt\\
&=\wh{\lambda} \cdot \int_T a(t)t^{\kappa(\lambda)} t\cdot s \, dt.
\end{align*}
Hence
\begin{align*} 
\|a\cdot \wh{\lambda} \cdot s\|^2&=\int_{T \times T} t_1^{\kappa(\lambda)}t_2^{\kappa(\lambda)} a(t_1)a(t_2)\big(t_1\cdot s,t_2\cdot s\big)_{L^2(\Y,S)}\, dt_1\,dt_2\\
&=f^\wedge\big(\kappa(\lambda),\kappa(\lambda)\big),
\end{align*}
where $f^\wedge$ is the (inverse) Fourier transform of
\[ f(t_1,t_2)=a(t_1)a(t_2)\big(t_1\cdot s,t_2\cdot s\big)_{L^2(\Y,S)}.\]
The result now follows because $f\colon T\times T \rightarrow \bC$ is smooth, hence its Fourier coefficients go to zero as $|\lambda| \rightarrow \infty$.
\end{proof}
\begin{remark}
\label{rem:Canonical}
In \cite{LMSspinor} we constructed a canonical spinor module on $\Y$, which in fact was the pullback to $\Y$ of a $\wh{LH}$-equivariant spinor module for the vector bundle $\pi^\ast TM$ on $\M$; the relevant central extension of $LH$ here is the \emph{spin central extension} (cf. \cite{PressleySegal,FHTII}).  If $H$ is semisimple, the resulting central extension of $T \times \Lambda$ satisfies \eqref{eqn:CommutationRelation}, and the conditions in Proposition \ref{prop:LambdaTAdmissible} are met.  One can also twist by a prequantum line bundle, cf. \cite{AMWVerlinde} for a discussion of prequantum line bundles on Hamiltonian loop group spaces.  If $H$ is not semi-simple, one may still be able to meet the conditions in Proposition \ref{prop:LambdaTAdmissible} by twisting the canonical spinor module with a suitable line bundle.  These examples were our principal motivation.
\end{remark}

\subsection{K-homology class associated to $\Y$.}\label{sec:KHomGlobTrans}
Choose a $T\times \Lambda$-invariant complete Riemannian metric $g$ on $\Y$.  We have the following immediate consequence of Theorem \ref{thm:GammaKAdmissible}, Corollary \ref{cor:Refine} and Proposition \ref{prop:LambdaTAdmissible}.
\begin{corollary}
\label{cor:KHomologyTrans}
Let $S \rightarrow \Y$ be a $G$-equivariant $\bZ_2$-graded Hermitian vector bundle, where $G$ is a central extension of $T \times \Lambda$ satisfying \eqref{eqn:CommutationRelation}.  Let $\st{D}$ be a $G$-equivariant, odd, symmetric $1^{st}$-order elliptic operator with finite propagation speed acting on sections of $S$.  Then $\st{D}$ defines a class $[\st{D}] \in \K^0(G \ltimes C_0(Y))$.
\end{corollary}
\begin{remark}
A motivating example is the spin-c Dirac operator (Example \ref{ex:DiracOp}) acting on sections of the canonical spinor module on $\Y$ (possibly twisted by a suitable prequantum line bundle) of Remark \ref{rem:Canonical}.
\end{remark}
As explained in Definition \ref{def:CapProd}, we obtain elements of $R^{-\infty}(T)$ by restricting to $T$ and taking cap products with classes $\st{x} \in \K^0_T(Y)$:
\begin{equation} 
\label{eqn:CapProdT}
\st{x} \cap \iota_T^\ast[\st{D}] \in \K^0(C^\ast(T))=R^{-\infty}(T).
\end{equation}

An element $\chi \in R^{-\infty}(T)$ determines a unique multiplicity function
\[ m \colon \Lambda^\ast \rightarrow \bZ \]
giving the coefficients of the formal expansion of $\chi$ in terms of the irreducible characters for $T$.  The multiplicity functions obtained in \eqref{eqn:CapProdT} are invariant under the action of $\kappa(\Lambda)$ on $\Lambda^\ast$ by translations.  This is a consequence of the $T \ltimes \hLambda$-equivariance of $[\st{D}]$ (and of $\st{x}$; note that $\st{x}$ can be promoted to a $G=T\ltimes \wh{\Lambda}$-equivariant element, since $\wh{\Lambda}$ acts trivially on $Y$), and the commutation relations \eqref{eqn:CommutationRelation}.

\subsection{Weyl anti-symmetry.}\label{sec:WeylAnti}
The cap products $\st{x} \cap \iota_T^\ast [\st{D}]$ encode K-theoretic invariants of the symplectic reduced spaces, analogous to the way twisted Duistermaat-Heckman distributions encode cohomological invariants of the reduced spaces.  Of particular interest is an analogue of the ordinary Duistermaat-Heckman measure, which we now explain.  

Let $\n_-$ (resp. $\n_+$) denote the direct sum of the negative (resp. positive) root spaces.  Thus 
\[ (\h/\t)_{\bC} \simeq \n_- \oplus \n_+. \]
Given $\xi \in (\h/\t)_{\bC}$, let $\xi_-$ (resp. $\xi_+$) be its component in $\n_-$ (resp. $\n_+$), hence for $\xi \in \h/\t$ real one has $\xi_+=\ol{\xi_-}$ (complex conjugation in $(\h/\t)_{\bC}$).  Consider the $T$-equivariant morphism of vector bundles over $\st{B}_q(\h/\t)$ given by
\[ \st{b} \colon \st{B}_q(\h/\t) \times \wedge^{\tn{even}}\n_- \rightarrow \st{B}_q(\h/\t) \times \wedge^{\tn{odd}}\n_-, \qquad \st{b}(\xi)=\sqrt{2}\big(\varepsilon(\xi_-)-\iota(\xi_+)\big)\]
where $\varepsilon$ (resp. $\iota$) denotes exterior multiplication (resp. contraction with respect to the bilinear form on $(\h/\t)_{\bC}$); this is the formula for the action of $\xi \in \h/\t$ on the $\Cl(\h/\t)$-module $\wedge \n_-$.  The morphism $\st{b}$ determines a K-theory class
\[ \Bott(\n_-) \in \K^0_T(\st{B}_q(\h/\t)),\]
called the \emph{Bott class}, which generates $\K^0_T(\st{B}_q(\h/\t))$ as a free $\K^0_T(\pt)=R(T)$-module of rank $1$.  Note that the construction depends on the choice of a complex structure (i.e. $\n_-$), and so does the corresponding K-theory element.

Recall $\phi^{\h/\t} \colon Y \rightarrow \st{B}_q(\h/\t)$ is a proper $N_H(T)$-equivariant map.  Let $\X$ denote the (possibly singular) closed subset $(\phi^{\h/\t})^{-1}(0)=\Phi_{\M}^{-1}(\t)$.
\begin{definition}
\label{def:QuantX}
Let $[\X]\in \K^0_T(Y)$ denote the pullback under $\phi^{\h/\t}$ of the Bott class $\Bott(\n_-) \in \K^0_T(\st{B}_q(\h/\t))$.  The index pairing $Q(\X,\st{D})$ is
\[ Q(\X,\st{D})=[\X] \cap \iota_T^\ast [\st{D}] \in \K^0(C^\ast(T))\simeq R^{-\infty}(T).\]
\end{definition}
\begin{remark}
The reason for the notation $[\X]$ is that this class plays the role of a Poincare dual to $\X$ in K-theory.  Thus studying the index \emph{pairing} in Definition \ref{def:QuantX} is a substitute for studying the index of a Dirac operator on the (possibly) singular space $\X$ (and accordingly we think of $Q(\X,\st{D})$ as the `quantization' of $\X$).  This analogy becomes precise when $0$ is a regular value of $\phi^{\h/\t}$ (implying $\X$ is smooth): in this case we may construct $\Y$ to be diffeomorphic to $\X \times (\h/\t)$, and it follows that $Q(\X,\st{D})$ is the $T$-index of a Dirac operator on $\X$ for the induced spinor module $S_{\X}=\Hom_{\Cl(\h/\t)}(\wedge \n_+,S)$.  For further context see the discussion in \cite[Section 6]{LMSspinor}.  The analogous construction for the Duistermaat-Heckman measure (involving cohomology instead of K-theory) is discussed in \cite{DHNormSquare}.  
\end{remark}

The Bott element, and hence also $Q(\X,\st{D})$, has an additional anti-symmetry property under the action of the normalizer $N_H(T)$ of $T$ in $H$.  To deduce the consequences of this symmetry, we will use material from Appendix \ref{sec:GrpAut}.  We now assume the vector bundle $S$ and $1^{st}$ order elliptic operator $\st{D}$ are equivariant for a central extension $G$ of $N_H(T)\ltimes \Lambda$ (rather than $T \times \Lambda$).  Indeed this is the case for the spin-c Dirac operator and the canonical spinor module of Remark \ref{rem:Canonical} (the latter comes from a $\wh{LH}$-equivariant spinor module on $\M$), as well as twists by various equivariant line bundles.  Then the discussion in Section \ref{sec:KHomGlobTrans} carries through without change, and we obtain a class $[\st{D}] \in \K^0(G\ltimes C_0(Y))$, where $G$ now denotes a central extension of $N_H(T)\ltimes \Lambda$.

Let $w \in W=N_H(T)/T$, and choose a lift $h \in N_H(T)$.  The element $w$ determines an automorphism of $T$, given by $t \mapsto t^w=hth^{-1}$.  Let $h \in N_H(T)$ act on $f \in C_0(\h/\t)$ by $f^h:=f\circ \Ad_{h^{-1}}$.  Using the results of Appendix \ref{sec:GrpAut}, $w$ determines an automorphism $\tau_w$ of $\K^0_T(\h/\t)=\KK_T(\bC,C_0(\h/\t))$.

\begin{proposition}
\label{prop:BottWeyl}
Under the action of $\tau_w$ we have
\[ \tau_w\big(\Bott(\n_-)\big)=(-1)^{l(w)}\bC_{\rho-w\rho}\otimes \Bott(\n_-),\]
where $l(w)$ is the length of the Weyl group element $w$.  Since $\phi^{\h/\t}$ is $N_H(T)$-equivariant, it follows that $\tau_w([\X])=(-1)^{l(w)}\bC_{\rho-w\rho}\otimes [\X]$.
\end{proposition}
\begin{proof}
By Bott periodicity $\K^0_T(\h/\t)$ is a free rank 1 module over $\K^0_T(\pt)$ generated by the Bott element $\Bott(\n_-)$.  Hence there must exist an element $V \in \K^0_T(\pt)\simeq R(T)$ such that
\[ \tau_w\big(\Bott(\n_-)\big)=V \otimes \Bott(\n_-) \in \K^0_T(\h/\t).\]
Thinking of $V$ as a $\bZ_2$-graded representation of $T$, $V$ must have degree $(-1)^{l(w)}$ since $w$ changes the orientation (hence the grading) by $(-1)^{l(w)}$; this explains the factor $(-1)^{l(w)}$ in the formula.  To determine the $T$-action, it is enough to consider the pullback of $\Bott(\n_-)$ to $\{0 \} \subset \h/\t$, which is the graded $T$-representation $\wedge \n_- \in \K^0_T(\pt)\simeq R(T)$.  Since $h$ fixes $0 \in \h/\t$ one just needs to determine the action of $\sigma_w^{-1}$ (see Appendix \ref{sec:GrpAut} for this notation).  Recall $\wedge \n_- \otimes \bC_\rho$ has weights which are symmetric with respect to the $W$-action, hence
\[ \sigma_w^{-1}(\wedge \n_- \otimes \bC_\rho)=\wedge \n_- \otimes \bC_\rho,\]
which implies
\[ \sigma_w^{-1}(\wedge \n_-)=\wedge \n_- \otimes \bC_{\rho-w\rho} \in \K^0_T(\pt).\]
\end{proof}

Using the results of Appendix \ref{sec:GrpAut} we now find:
\begin{align*}
\tau_w^{\bC} \otimes Q(\X,\st{D}) &= \tau_w^{\bC}\otimes j_T([\X])\otimes (\tau_w^{C_0(Y)})^{-1}\otimes \tau_w^{C_0(Y)}\otimes \iota_T^\ast[\st{D}]\\
&=j_T(\tau_w[\X])\otimes \tau_w^{C_0(Y)}\otimes \iota_T^\ast[\st{D}]\\
&=(-1)^{l(w)}\bC_{\rho-w\rho}\otimes j_T([\X])\otimes \tau_w^{C_0(Y)}\otimes \iota_T^\ast[\st{D}]\\
&=(-1)^{l(w)}\bC_{\rho-w\rho}\otimes j_T([\X])\otimes \iota_T^\ast[\st{D}]\\
&=(-1)^{l(w)}\bC_{\rho-w\rho}\otimes Q(\X,\st{D}),
\end{align*}
where in the second, third, fourth lines we used Propositions \ref{prop:DescentMapTwist2}, \ref{prop:BottWeyl}, \ref{prop:DescentMapTwist3} respectively.  Together with the $\kappa(\Lambda)$-invariance, it follows that the multiplicity function 
\[ m \colon \Lambda^\ast \rightarrow \bZ \]
of $ Q(\X,\st{D})$ is \emph{alternating} under the $\rho$-shifted action of the affine Weyl group determined by $\kappa$.  Explicitly, for $w=(\ol{w},\eta) \in W \ltimes \Lambda=W_{\aff}$, one has
\begin{equation} 
\label{eqn:ShiftedAction}
m(w\cdot (\lambda+\rho)-\rho)=(-1)^{l(w)}m(\lambda),
\end{equation}
where $w\cdot \xi=\ol{w}\xi+\kappa(\eta)$.

\subsection{Special case: $H$ simple and simply connected.}
In this case the $U(1)$ central extensions of $LH$ are classified by an integer $k$ called the \emph{level}.  If $k \ne 0$ then the corresponding extension obeys \eqref{eqn:CommutationRelation}.  Moreover, the spin central extension mentioned in Remark \ref{rem:Canonical} is known to be at level equal to the dual Coxeter number $\hvee$ of $\h$ (cf. \cite{PressleySegal,LMSspinor}).

Let $\wh{LH}$ be the extension corresponding to the generator $k=1$; it corresponds to the map $\kappa \colon \Lambda \rightarrow \Lambda^\ast$ induced by the \emph{basic inner product}, the unique invariant inner product on $\h$ such that the squared lengths of the short co-roots is $2$.  There is an interesting class of representations of $\wh{LH}$, known as \emph{positive energy representations}, which have a rich theory parallel to the representation theory of compact Lie groups, cf. \cite{KacBook,PressleySegal}.  By definition, a positive energy representation is one that admits an extension to a representation of the semi-direct product $S^1_{\rot}\ltimes \wh{LH}$ ($S^1$ acts on $\wh{LH}$ by a lift of its action on $LH$ by rigid rotations), such that the weights of the $S^1_{\rot}$ action are bounded below and of finite multiplicity.  Positive energy representations have formal characters $\chi \in R^{-\infty}(T\times S^1_{\rot})$ which can be computed by a version of the Weyl character formula.  For each $k=1,2,3,...$, there is an analogue of the representation ring, the \emph{level} $k$ \emph{fusion ring} $R_k(H)$, with basis given by the (finite) set of irreducible positive energy representations having central weight $k$ and minimal $S^1_{\rot}$ weight normalized to be $0$.

Let $k$ be a positive integer, and let $\kappa_{k+\hvee}=(k+\hvee)\kappa \colon \Lambda \rightarrow \Lambda^\ast$ be the map determined by $(k+\hvee)$ times the basic inner product.  There is a 1-1 correspondence between elements of $R^{-\infty}(T)$ alternating under the $\rho$-shifted action of the affine Weyl group determined by $\kappa_{k+\hvee}$ (equation \eqref{eqn:ShiftedAction}), and elements of the level $k$ fusion ring $R_k(H)$ (cf. \cite[Chapter 14]{PressleySegal}, \cite{KacBook,FHTII}) given as follows: for $\chi \in R^{-\infty}(T\times S^1_{\textnormal{rot}})$ the formal character of an element of the fusion ring, the corresponding element of $R^{-\infty}(T)$ is
\[ \big(\chi \cdot \Delta \big)|_{q=1} \]
where $\Delta=\prod_{\bm{\alpha}>0}(1-e_{-\bm{\alpha}})$ is the \emph{Weyl-Kac denominator} and $q \in S^1_{\textnormal{rot}}$.

Let $S$ be the canonical level $\hvee$ spinor module on $\Y$ (Remark \ref{rem:Canonical}) twisted by a level $k>0$ prequantum line bundle $L$, and let $\st{D}$ be a spin-c Dirac operator acting on sections of $S$.
\begin{definition}
\label{def:QuantizeM}
We define the level $k$ \emph{quantization} of $\M$ to be the element $Q(\M,L) \in R_k(H)$ such that
\[ \big(Q(\M,L)\cdot \Delta \big)|_{q=1}=Q(\X, \st{D}) \in R^{-\infty}(T).\]
\end{definition}

\begin{remark}
As mentioned in the introduction, one can do a completely analogous construction to \ref{def:QuantizeM} in the case of a compact Hamiltonian $H$-space $\mu \colon M \rightarrow \h^\ast$ and the (possibly singular) closed subset $X=\mu^{-1}(\t^\ast)$, and the result is equivalent to the usual definition in terms of the index of an $H$-equivariant Dirac operator on $M$.  See also the discussion in \cite[Section 6]{LMSspinor}. 
\end{remark}

\begin{remark}[\emph{Non-simply connected $H$}]
If $H$ is simple (or semi-simple) but not necessarily simply connected, one has a generalization of Definition \ref{def:QuantizeM}.  Using the canonical spin-c structure of Remark \ref{rem:Canonical} (optionally twisted by a prequantum line bundle), one obtains a $W_{\aff}$ anti-symmetric element $Q(\X,\st{D}) \in R^{-\infty}(T)$ as explained in Section \ref{sec:WeylAnti}.  In general the result is the $q=1$ restriction of the Weyl-Kac numerator (as in \ref{def:QuantizeM}) of a (unique) $\bZ_2$-\emph{graded} positive energy representation (cf. \cite[Chapter 14]{PressleySegal} for the general character formula, and \cite{FHTII,FHTIII} for further discussion of graded representations and gradings of loop groups).  The formal character of a $\bZ_2$-graded representation is the difference of the formal characters of the two homogeneous subspaces (in general, the homogeneous subspaces will themselves only be representations of the connected component of the identity in $LH$).
\end{remark}

\subsection{A cycle for the cap product.}
We now return to a slightly more general setting and describe an unbounded cycle representing the cap product $\st{x}\cap \iota_T^\ast[\st{D}]$, where $\st{x} \in \K^0_T(Y)$, and $\st{D}$ is a spin-c Dirac operator (Example \ref{ex:DiracOp}) acting on sections of a spinor module $S$ satisfying the conditions in Corollary \ref{cor:KHomologyTrans}.  References in which similar products are studied include \cite{KucerovskyCallias, BunkeRelativeIndex} and \cite[Exercise 11.8.14]{HigsonRoe}.  Our work is simpler if we choose the metric on $Y$ and the cycle representing $\st{x}$ with a little care.

\subsubsection{Metrics on $Y$, $\Y$.}
Recall from Section \ref{sec:GlobTrans} that $Y=\Phi^{-1}(U)$, where
\[ U \simeq T \times \st{B}_q(\h/\t).\]
Redefining $Y$ to be smaller if necessary, we may assume $q$ is a regular value of the function $r=|\phi^{\h/\t}|$.  Properness implies there is a small interval $(q-2\epsilon,q+2\epsilon)$, $0<2\epsilon<q$ such that
\[ r^{-1}(q-2\epsilon,q+2\epsilon) \simeq Q \times (q-2\epsilon,q+2\epsilon), \qquad Q=r^{-1}(q).\]
The closure $\ol{Y}$ of $Y$ in $M$ is the smooth manifold with boundary $\partial \ol{Y}=Q$. Let
\[ \Cyl_Q=r^{-1}\big((q-\epsilon,q)\big),\]
and define $x \colon \Cyl_Q \rightarrow (1, \infty)$ by
\[ e^{-x+1}=\epsilon^{-1}(q-r).\]
Note $r \rightarrow q$ corresponds to $x \rightarrow \infty$, while $r \rightarrow q-\epsilon$ corresponds to $x \rightarrow 1$.  Extend $x$ to a smooth function
\begin{equation} 
\label{eqn:Extendx}
x \colon Y \rightarrow [0,\infty) 
\end{equation}
such that $x^{-1}(1,\infty)=\Cyl_Q$.  Using a partition of unity, one can construct a complete $N_H(T)$-invariant Riemannian metric $g$ on $Y$ such that
\[ g|_{\Cyl_Q}=dx^2+g_Q,\]
where $g_Q$ is a metric on the compact manifold $Q$; the open set $\Cyl_Q$ is a cylindrical end for the metric $g$.  The pullback of $g$ to $\Y$ is a $N_H(T)\ltimes \Lambda$-invariant complete metric on $\Y$.

\subsubsection{Unbounded cycles for $\K_T^0(Y)$.}\label{sec:UnboundedRepresentative}
An element $\st{x} \in \K^0_T(Y)$ may be represented by a pair $(E,\theta)$ consisting of a $\bZ_2$-graded $T$-equivariant Hermitian vector bundle $E=E^+\oplus E^-$ and an odd, self-adjoint $T$-equivariant bundle endomorphism $\theta$, which is invertible outside a compact subset of $Y$.

We may assume $\theta$ is bounded and $\theta^2=\id$ on $\Cyl_Q$.  After adding a vector bundle $E^\prime$ to $E^+$ and $E^-$ such that $E^+\oplus E^\prime$ is trivial and extending $\theta$ to $E^\prime$ by the identity (this does not change the class in $\K_T^0(Y)$), we can assume $E^{\pm}|_{\Cyl_Q} \simeq \Cyl_Q \times \bC^k$ and 
\[ \theta|_{\Cyl_Q}\colon \Cyl_Q \times \bC^k \rightarrow \Cyl_Q \times \bC^k \] 
is the identity map.  We may further assume that the $T$-action on $E|_{\Cyl_Q} \simeq [q,\infty) \times E|_Q$ is the product action; this follows, for example, from the rigidity of compact group actions on compact manifolds, applied to the unit sphere bundle in $E|_Q$.  In terms of bounded cycles for KK-theory, the corresponding element is $[(C_0(Y,E),\theta)]$, where the $C_0(Y)$-valued inner product on $C_0(Y,E)$ is given by the Hermitian structure.

\begin{definition}
\label{def:fcnf}
Let
\[ f\colon [0,\infty) \rightarrow [1,\infty) \]
be a smooth, monotone non-decreasing function, equal to $1$ on a neighborhood of $[0,1]$, such that $f(t) \rightarrow +\infty$ and $f(t)^{-2}f^\prime(t) \rightarrow 0$ as $t \rightarrow +\infty$.  Precomposing with $x \colon Y \rightarrow [0,\infty)$ (see \eqref{eqn:Extendx}) we obtain a function $f \circ x$ on $Y$; to keep notation simple we will denote this composition by $f$.
\end{definition}

The bundle endomorphism $f\theta$ is odd, self-adjoint by construction, and it is not difficult to check that it is also regular as an unbounded operator on the Hilbert $C_0(Y)$-module $C_0(Y,E)$.  Moreover $(1+f^2\theta^2)^{-1}$ vanishes at infinity by construction, so defines a compact operator on the Hilbert $C_0(Y)$-module $C_0(Y,E)$.  Hence the pair $(E,f\theta)$ is an unbounded cycle representing $\st{x} \in \K^0_T(Y)$.

\subsubsection{The $\KK$ product.}\label{sec:CapProduct}
Let $G$, $S$ and $\st{D}$ be as in Corollary \ref{cor:KHomologyTrans}, and $[\st{D}] \in \K^0(G\ltimes C_0(Y))$ the corresponding K-homology class.  Let $(E,f\theta)$ be a cycle representing $\st{x} \in \K^0_T(Y)$, as constructed in Section \ref{sec:UnboundedRepresentative}.

Under the descent map $j_T$, the Hilbert $C_0(Y)$ module $C_0(Y,E)$ is sent to the $C^\ast(T)$-$T\ltimes C_0(Y)$ bimodule denoted $T \ltimes C_0(Y,E)$.  The following lemma describes the Hilbert space for a cycle representing the cap product $\st{x}\cap \iota_T^\ast [\st{D}]$.
\begin{lemma}
\label{LemmaHilbSpace}
There is an $T$-equivariant isomorphism of $\bZ_2$-graded Hilbert spaces
\[ T\ltimes C_0(Y,E) \wh{\otimes}_{T\ltimes C_0(Y)} L^2(\Y,S) \simeq L^2(\Y,E \wh{\otimes} S),\]
where the left-hand-side is a completed, $\bZ_2$-graded tensor product of $T \ltimes C_0(Y)$-modules.\footnote{Here and wherever possible below we have written $E$ instead of $\pi^\ast E$.}
\end{lemma}
\begin{proof}
Let $e \in T \ltimes C_0(Y,E)$ and $s \in L^2(\Y,S)$.  The isomorphism is given on the dense subspace of elements of the form $e\wh{\otimes} s$ by
\[ e \wh{\otimes} s \mapsto \T_e(s)= \int_T e(t)\wh{\otimes} t\cdot s \,\,dt \]
where $e(t) \in C_0(Y,E)$ and $t\cdot s \in L^2(\Y,S)$ is the section $(t\cdot s)(y)=t(s(t^{-1}y))$.  It is straightforward but tedious to verify that this intertwines the inner products, so extends to an isometric isomorphism between Hilbert spaces.  It is also straightforward to verify that this intertwines the $T$-actions.
\end{proof}
Choose a $T$-invariant Hermitian connection $\nabla^E$ on $E$, which extends the trivial connection on $\Cyl_Q$.  Using $\nabla^E$, we couple the spin-c Dirac operator $\st{D}$ to $E$ to obtain an operator $\st{D}^E$ acting on sections of $E\wh{\otimes} S$; explicitly, in terms of a local orthonormal frame $v_i$, $i=1,...,\dim(\Y)$
\[ \st{D}^E(e\wh{\otimes} s)=(-1)^{\deg(e)}\big(\nabla^E_{v_i}e\wh{\otimes} \c(v_i)s+e\wh{\otimes} \st{D}s\big),\]
where the $(-1)^{\deg(e)}$ appears because of the Koszul sign rule, as $\c(v_i)$ has odd degree.  Let $\nabla^{\End(E)}$ denote the induced connection on $\End(E)$, defined by the equation
\[ \nabla^{\End(E)}_v\theta=[\nabla^E_v,\theta]\]
where both sides are viewed as operators on smooth sections of $E$.
\begin{lemma}
\label{lem:KKcycle}
The pair $(L^2(\Y,E\wh{\otimes} S),\st{D}^E+f\theta)$ is a cycle for $\K^0(C^\ast(T))$. 
\end{lemma}
\begin{proof}
The argument is an adaptation of the proof of Theorem \ref{thm:GammaKAdmissible}.  Choose a \emph{finite} open covering $V_j$, $j \in \J$ of $T$ satisfying conditions (a), (b) in Section \ref{sec:PartUnity} with respect to the covering space $\exp \colon \t \rightarrow T$ and covering group $\Gamma=\Lambda$.  Let $U_j^\prime=r_j(V_j)$ be lifts to $\t$, and $U_j=\phi^{-1}(U_j^\prime) \subset \Y$ (see \eqref{eqn:Pullback} for the definition of $\phi$).  For $\lambda \in \Lambda$, define $\sigma_j^{\lambda}$, $\nu_j^\lambda$, $\ti{\nu}_j$, $\ti{\sigma}_j$ as in Section \ref{sec:PartUnity}.  Lemma \ref{lem:finitesum} works as before.  The proof of Lemma \ref{prop:compactness} requires modification, because $\nu_j$, $\sigma_j$ are no longer compactly supported (the map $\phi$ is not proper), so we cannot use Proposition \ref{RellichLemma} to conclude that the operator
\[ T=\sum_{\gamma^\prime \in \Gamma^\prime}\sigma^{\gamma^\prime}_j \chi(\st{D}^E+f\theta) \nu_j \]
is compact.  Instead we will use the main result of Appendix \ref{sec:SchrodingerOps}.

Squaring $\st{D}^E+f\theta$ we find
\[ (\st{D}^E)^2+[\st{D}^E,f\theta]+f^2\theta^2. \]
Since $\theta$ is self-adjoint, $\theta^2(y)$ is a non-negative endomorphism of $E$ for each $y \in Y$, and moreover $\theta^2|_{\Cyl_Q}=1$.  Let $\vartheta^2(y)$ be the minimum eigenvalue of $\theta^2(y)$.  In terms of a local orthonormal frame $v_1,...,v_{\dim(\Y)}$ the graded commutator $[\st{D}^E,f\theta]$ is the bundle endomorphism
\[ -\theta \wh{\otimes} \c(df)-\nabla_{v_i}^{\End(E)}\theta \wh{\otimes}\c(v_i) \]
of $E\wh{\otimes}S$.  Thus
\[ A=1+\big(\st{D}^E+f\theta\big)^2 \ge \big(\st{D}^E\big)^2+\pi^\ast V \]
where
\[ V=-|df|\cdot |\theta|-f|\nabla^E\theta|+f^2\vartheta^2+1.\]
We claim that $V$ is proper and bounded below on $Y$.  Since $Y \setminus \Cyl_Q$ is compact, it suffices to consider the restriction of $V$ to $\Cyl_Q \simeq Q \times (1,\infty)$ (in terms of the coordinate $x$), which is  
\[ V|_{\Cyl_Q}=-f^\prime(x)+f(x)^2+1, \qquad x \in [1,\infty) \]
and this is proper and bounded below by assumption (Definition \ref{def:fcnf}).

The pullback $\pi^\ast V$ to $\Y$ is no longer proper, but it is proper (and bounded below) on the support of $\sigma_j^{\gamma^\prime}$.  By Proposition \ref{CompactnessBoundedBelow}, the operator
\[ \sigma_j^{\gamma^\prime} A^{-1} \]
is compact.  Any function $\chi \in C_0(\bR)$ can approximated in norm by a product of $(1+x^2)^{-1}$ and a bounded continuous function.  Thus
\[ \sigma_j^{\gamma^\prime}\chi(\st{D}^E+f\theta) \]
is compact for any $\chi \in C_0(\bR)$.  The rest of the proof of Lemma \ref{prop:compactness} is as before.  The proof of Theorem \ref{thm:GammaKAdmissible} proceeds as before, except the final step is easier: since the index set $\J$ is \emph{finite}, the result follows from the fact that a finite sum of compact operators is compact.
\end{proof}
\begin{theorem}
\label{thm:CapProductResult}
Let $\st{x} \in \K^0_T(Y)$ be represented by a cycle $(E,f\theta)$ as in Section \ref{sec:UnboundedRepresentative}, and let $[\st{D}^E+f\theta]$ be the cycle in Lemma \ref{lem:KKcycle}.  Then
\[ [\st{D}^E+f\theta]=\st{x}\cap \iota_T^\ast [\st{D}] \in \K^0(C^\ast(T)). \]
\end{theorem}
\begin{proof}
We will verify the conditions in Proposition \ref{prop:SuffCond}.  Consider the dense set of $e \in T\ltimes C_0(Y,E)$ of the form $e=a\sigma$, where $a \in C(T)$ and $\sigma \in C_c^\infty(Y,E)$.  We must show that the operator
\[ \T_e \st{D}-(-1)^{\deg(e)}(\st{D}^E+f\theta)\T_e \]
is bounded, where $\T_e\colon L^2(\Y,S) \rightarrow L^2(\Y,E \wh{\otimes} S)$ is defined in the proof of Lemma \ref{LemmaHilbSpace}.  Since $\sigma$ has compact support, $f$ is bounded on the support of $e$, hence $f\theta \T_e$ is bounded.  In terms of a local orthonormal frame $v_1,...,v_{\dim(\Y)}$, the difference $\T_e \st{D}-(-1)^{\deg(e)}\st{D}^E\T_e$ is the operator $L^2(\Y,S) \rightarrow L^2(\Y,E\wh{\otimes}S)$ given by 
\begin{equation}
\label{eqn:ProductCommutator}
s \mapsto -\int_T dt \,\, a(t)\pi^\ast\nabla^E_{v_i}\sigma\wh{\otimes}\c(v_i)(t\cdot s).
\end{equation}
Since $\sigma \in C_c^\infty(Y,E)$ and $\pi^\ast \nabla^E \sigma$ is $\Lambda$-invariant, the norm of $\pi^\ast \nabla^E\sigma$ is bounded.  Thus \eqref{eqn:ProductCommutator} is a bounded operator.  Likewise $\T_e^\ast(f\theta)$ is bounded, while $\st{D}\T_e^\ast-(-1)^{\deg(e)}\T_e^\ast \st{D}^E$ is given by a similar expression to \eqref{eqn:ProductCommutator}, and so is bounded also.

The operators $\st{D}^E+f\theta$ and $f\theta$ have a common core consisting of smooth compactly supported sections of $E \wh{\otimes} S$.  We showed in the proof of Lemma \ref{lem:KKcycle} that 
\begin{equation}
\label{eqn:SemiBded}
[\st{D}^E,f\theta]+f^2\theta^2
\end{equation}
is semi-bounded below, say by $-c \le 0$, on the common core.  Let $\sigma_n \in C_c^\infty(Y,E\wh{\otimes}S)$ be a sequence of smooth compactly supported sections such that $\sigma_n \rightarrow 0$ and $(\st{D}^E+f\theta)\sigma_n \rightarrow 0$ in $L^2$.  Since
\[ \|\st{D}^E\sigma_n\|^2+\big(([\st{D}^E,f\theta]+f^2\theta^2+c)\sigma_n,\sigma_n\big)_{L^2}=\|(\st{D}^E+f\theta)\sigma_n\|^2+c\|\sigma_n\|^2 \rightarrow 0 \]
the semi-boundedness implies $\st{D}^E\sigma_n \rightarrow 0$, and hence $f\theta \sigma_n \rightarrow 0$ also.  This shows that the graph norm for $\st{D}^E+f\theta$ is stronger than that for $f\theta$, hence $\dom(\st{D}^E+f\theta) \subset \dom(f\theta)$.  The inequality in the semi-boundedness condition of Proposition \ref{prop:SuffCond} amounts to showing that the graded commutator $[\st{D}^E+f\theta,f\theta]=[\st{D}^E,f\theta]+2f^2\theta^2$ is semi-bounded below, but this follows from the semi-boundedness of \eqref{eqn:SemiBded} and $f^2\theta^2$.
\end{proof}

\subsection{Remarks on other approaches to the quantization problem.}\label{ssec:OtherApproaches}
In this subsection we assume $H$ is simple and simply connected.  To date, two other index or K-theoretic approaches to the `quantization' of general proper Hamiltonian $LH$-spaces have appeared in the literature. 

The earliest approach was due to Meinrenken, who studied the finite dimensional compact \emph{quasi-Hamiltonian} $H$-space $M=\M/\Omega H$.  $M$ comes naturally equipped with a `group-valued moment map' $\Phi \colon M \rightarrow H$.  Let $[M] \in \KK_H(\Cl(M),\bC)$ denote Kasparov's fundamental class, the class of the de Rham-Dirac operator acting on differential forms on $M$.  Meinrenken \cite{MeinrenkenKHomology} defined the quantization of $M$ to be a push-forward $(\Phi,\E)_\ast[M]$ to a suitable \emph{twisted K-homology} group of $H$.  The definition of the push-forward map requires constructing a suitable Morita morphism $\E$ \cite{DDDFunctor,MeinrenkenKHomology}, which is the heart of the construction.  

The Freed-Hopkins-Teleman Theorem relates the fusion ring $R_k(H)$, $k>0$ with the twisted K-homology of $H$ (for twist $k+\hvee \in \bZ \simeq H^3_H(H,\bZ)$, $\hvee$ the dual Coxeter number of $H$).  Thus it seems reasonable to define the quantization of $\M$ to be the image of $(\Phi,\E)_\ast[M]$ under the Freed-Hopkins-Teleman isomorphism.  This definition turns out to have many desirable properties \cite{MeinrenkenKHomology}.

In \cite{LoizidesGeomKHom} the first author proved that Meinrenken's approach is equivalent to the approach developed here: it was shown that $Q(\M,L)$ as in Definition \ref{def:QuantizeM} coincides with the image of $(\Phi,\E)_\ast[M]$ under the Freed-Hopkins-Teleman isomorphism.  To summarize the approach of \cite{LoizidesGeomKHom}, an `index' map $\scr{I}$ was constructed from the twisted K-homology of $H$ to the formal completion $R^{-\infty}(T)$, using $C^\ast$-algebra methods.  The fairly explicit description of this map in terms of cycles makes it straight-forward to verify (by checking on generators) that it is compatible with the Freed-Hopkins-Teleman isomorphism, in the sense that the image of any class in the twisted K-homology group is the Weyl-Kac numerator of the corresponding positive energy representation.  

On the other hand, when $\scr{I}\big((\Phi,\E)_\ast[M]\big)$ is computed, one obtains exactly the $T$-index of the operator $\st{D}^E+f\theta$ on $\Y$ with $[(E,\theta)]$ being the Bott element, as in Definition \ref{def:QuantizeM} and Theorem \ref{thm:CapProductResult}.  The way in which the $\Lambda$-covering space $\Y$ of $Y \subset M$ appears turns out to be quite natural and conceptually clear: the $\KK$-theory cycle describing the push-forward $(\Phi,\E)_\ast[Y]$ ($\scr{I}$ involves a restriction to $Y \subset M$) naturally involves a Dirac-type operator acting on a smooth Hilbert bundle over $Y$ with fibres $\ell^2(\Lambda)$ (tensored with a finite dimensional bundle), and this is equivalent to a Dirac-type operator acting on a finite dimensional bundle over the covering space.  See \cite{LoizidesGeomKHom} for details.

A different approach to quantization for Hamiltonian $LH$-spaces due to the second author appeared in \cite{SongDiracLoopGroup}.  In this approach, an operator acting on sections of an infinite dimensional bundle over the quasi-Hamiltonian space $M$ was constructed, whose index was directly a positive energy representation of $LH$ (strictly speaking a formal difference of two such).  Thus the intermediaries appearing in the other two approaches (the Freed-Hopkins-Teleman Theorem in \cite{MeinrenkenKHomology}, and multiplication by the Weyl-Kac denominator in Definition \ref{def:QuantizeM}) were not needed.  Unfortunately the definition of the operator in \cite{SongDiracLoopGroup} was quite complicated; for example, the operator was constructed locally in cross-sections, and patched together using a partition of unity.  The definition also involved combining geometric Dirac operators on submanifolds of $M$ with Kostant's algebraic relative `cubic' Dirac operator for loop groups.  This has made it difficult to compare with the other approaches.  We hope to return to this question in the future.  Related work was done by D. Takata \cite{Takata1,Takata2} for the case that $H$ is a torus.

\appendix

\section{Group automorphisms and the descent map.}\label{sec:GrpAut}
Throughout this section $G$ will be a locally compact group equipped with a left-invariant Haar measure, $\sigma \in \Aut(G)$ will be a group automorphism preserving the Haar measure, and $\langle \sigma \rangle \ltimes G$ the semi-direct product group.  We write $g \mapsto g^\sigma$ for the action of $\sigma$ on $g \in G$.

Let $A$ be a $\langle \sigma \rangle \ltimes G$-$C^\ast$ algebra; $A$ is also a $G$-$C^\ast$ algebra by restricting the group action.  Define the $G$-$C^\ast$ algebra $A^\sigma$ to be the $C^\ast$ algebra $A$ equipped with the new $G$-action $\pi^\sigma(g)=\pi(g^\sigma)$, where $\pi$ is the $G$-action on $A$.  Since $A$ is a $\langle \sigma \rangle \ltimes G$-$C^\ast$ algebra, there is an action map 
\[ \alpha^A_\sigma \colon A \rightarrow A.\]  
Viewed as a map $A \rightarrow A^\sigma$, $\alpha^A_\sigma$ is an isomorphism of $G$-$C^\ast$ algebras.  In particular, it induces maps
\[ (\alpha^A_\sigma)_\ast \colon \K_0^G(A) \rightarrow \K_0^G(A^\sigma), \qquad (\alpha^A_\sigma)^\ast \colon \K^0_G(A^\sigma) \rightarrow \K^0_G(A).\]
Below we usually omit the $\ast$ from the notation.

For any group homomorphism $\sigma \colon G_1 \rightarrow G_2$ one has a \emph{restriction homomorphism} (cf. \cite{KasparovNovikov})
\[ \sigma \colon \KK_{G_2}(A,B) \rightarrow \KK_{G_1}(A,B),\]
where $A$, $B$ become $G_1$-$C^\ast$ algebras by pre-composing the $G_2$ action with $\sigma$.  As a special case, if $\sigma \in \Aut(G)$ is a group automorphism, one obtains a map
\[ \sigma \colon \KK_G(A,B) \rightarrow \KK_G(A^\sigma,B^\sigma).\]

\begin{definition}
Let $A$, $B$ be $\langle \sigma \rangle \ltimes G$-$C^\ast$ algebras, with $\alpha_\sigma^A$, $\alpha_\sigma^B$ the corresponding action maps.  Define
\[ \tau_\sigma \colon \KK_G(A,B) \rightarrow \KK_G(A,B) \]
to be the composition
\[ \KK_G(A,B) \xrightarrow{\alpha_\sigma^B} \KK_G(A,B^\sigma) \xrightarrow{(\alpha_\sigma^A)^{-1}} \KK_G(A^\sigma,B^\sigma) \xrightarrow{\sigma^{-1}} \KK_G(A,B).\]
\end{definition}
In general $\tau_\sigma$ is not the identity map (we saw an example in Section \ref{sec:KHomGlobTrans}), but it does act as the identity on the image of the restriction map $\KK_{\langle \sigma \rangle \ltimes G}(A,B) \rightarrow \KK_G(A,B)$.

We next describe the relationship between $\tau_\sigma$ and the descent map $j_G$.  The $G$-equivariant $^\ast$-homomorphism $\alpha^A_\sigma \colon A \rightarrow A^\sigma$ induces a $^\ast$-homomorphism
\begin{equation}
\label{eqn:descentalpha}
\alpha^A_\sigma \colon G \ltimes A \rightarrow G \ltimes A^\sigma 
\end{equation}
(on $C_c(G,A)$ it is given by applying $\alpha^A_\sigma$ point-wise).  Recall that a $^\ast$-homomorphism $\alpha \colon A \rightarrow B$ determines an element $\alpha \in \KK(A,B)$, such that push-forward (resp. pull-back) by $\alpha$ in K-theory (resp. K-homology) are given by an appropriate $\KK$-product.  Thought of in this way, the corresponding element in $\KK(G\ltimes A,G\ltimes A^\sigma)$ is the image of $\alpha^A_\sigma \in \KK_G(A,A^\sigma)$ under the descent map $j_G$.

The group automorphism $\sigma \colon G \rightarrow G$ also induces a $^\ast$-homomorphism
\begin{equation} 
\label{eqn:descentsigma}
\sigma^A \colon G \ltimes A \rightarrow G \ltimes A^\sigma,
\end{equation}
given on $C_c(G,A)$ by $a \mapsto a^\sigma$, where $a^\sigma(g):=a(g^\sigma)$.
\begin{proposition}
\label{prop:DescentMapTwist}
Let $A$,$B$ be $\langle \sigma \rangle \ltimes G$-$C^\ast$ algebras.  For any $\st{x} \in \KK_G(A^\sigma,B^\sigma)$ one has the following equality in $\KK(G\ltimes A,G\ltimes B)$:
\[ j_G(\sigma^{-1}(\st{x}))=\sigma^A \otimes j_G(\st{x})\otimes (\sigma^B)^{-1}.\]
\end{proposition}
\begin{proof}
Let $(\E,\rho,F)$ be a cycle representing $\st{x}$, thus in particular $\E$ is an $(A^\sigma,B^\sigma)$-bimodule with $B^\sigma$-valued inner product $(\cdot,\cdot)$.  To avoid confusion between the $G$-$C^\ast$ algebras $B$ and $B^\sigma$, we will write all formulas in terms of the action map $(g^\prime,b)\mapsto g^\prime.b$ for $B$, \emph{not} $B^\sigma$.  Thus, for example, the $G\ltimes B^\sigma$-valued inner product for $j_G(\st{x})$ is expressed as
\[ (e_1,e_2)_{G\ltimes B^\sigma}(g)=\int_G (g_1^\sigma)^{-1}.\big(e_1(g_1),e_2(g_1g)\big)\,\,dg_1,\]
i.e. $(g_1^\sigma)^{-1}$ appears in the formula instead of $g_1^{-1}$.  The $(G\ltimes A,G\ltimes B)$-bimodule structure for $\sigma^A \otimes j_G(\st{x})\otimes (\sigma^B)^{-1}$ is given by the formulas
\[ (a\cdot e)(g)=\int_G a(g_1^\sigma)g_1.e(g_1^{-1}g)\,\,dg_1, \qquad (e\cdot b)(g)=\int_G e(g_1)g_1^\sigma.b\big((g_1^\sigma)^{-1}g^\sigma\big)\,\,dg_1, \]
and the $G\ltimes B$-valued inner product is
\[ (e_1,e_2)_{G\ltimes B}(g)=\int_G (g_1^\sigma)^{-1}.\big(e_1(g_1),e_2(g_1g^{\sigma^{-1}})\big)\,\, dg_1.\]
On the other hand, the $(G\ltimes A,G\ltimes B)$-bimodule structure for $j_G(\sigma^{-1}(\st{x}))$ is given by the formulas
\[ (a \star e)(g)=\int_G a(g_1)g_1^{\sigma^{-1}}.e(g_1^{-1}g)\,\,dg_1, \qquad (e \star b)(g)=\int_G e(g_1)g_1.b(g_1^{-1}g)\,\,dg_1, \]
and the $G\ltimes B$-valued inner product is
\[ (e_1,e_2)^{\star}_{G\ltimes B}(g)=\int_G g_1^{-1}.\big(e_1(g_1),e_2(g_1g)\big)\,\,dg_1.\]
The formulas are not the same, but there is a linear map
\[ e \in C_c(G,\E) \mapsto e^\sigma \in C_c(G,\E), \qquad e^\sigma(g):=e(g^\sigma) \]
which intertwines the bimodule structures and $C_c(G,B)$-valued inner products, i.e. $(a \star e)^\sigma=a\cdot e^\sigma$, $(e \star b)^\sigma=e^\sigma\cdot b$ and $(e_1^\sigma,e_2^\sigma)_{G\ltimes B}=(e_1,e_2)^{\star}_{G\ltimes B}$ (one uses the $\sigma$-invariance of the Haar measure).  Thus, the map extends to an isometric isomorphism between the completions, intertwining the bimodule structures.
\end{proof}

\begin{definition}
Let $A$ be a $\langle \sigma \rangle \ltimes G$-$C^\ast$ algebra.  Let
\[ \tau^A_\sigma=(\alpha^A_\sigma)^{-1} \circ \sigma^A \colon G\ltimes A \rightarrow G\ltimes A,\]
where $\alpha^A_\sigma$ (resp. $\sigma^A$) is as in \eqref{eqn:descentalpha} (resp. \eqref{eqn:descentsigma}).  Then $\tau^A_\sigma$ is an automorphism of $G\ltimes A$.  On $C_c(G,A)$ it is given by the formula
\[ \tau^A_\sigma(f)(g)=(\alpha^A_\sigma)^{-1}(f(g^\sigma)).\]
The corresponding element of $\KK(G\ltimes A,G\ltimes A)$ is the $\KK$-product
\[ \tau^A_\sigma=\sigma^A \otimes (\alpha^A_\sigma)^{-1} \in \KK(G\ltimes A,G\ltimes A).\]
\end{definition}
The following is a corollary of Proposition \ref{prop:DescentMapTwist}.
\begin{corollary}
\label{prop:DescentMapTwist2}
Let $A$, $B$ be $\langle \sigma \rangle \ltimes G$-$C^\ast$ algebras.  For any $\st{x} \in \KK_G(A,B)$,
\[ j_G(\tau_\sigma(\st{x}))=\tau^A_\sigma \otimes j_G(\st{x}) \otimes (\tau^B_\sigma)^{-1}.\]
\end{corollary}
Similar to $\tau_\sigma$, $\tau^A_\sigma$ acts trivially on elements which are $\langle \sigma \rangle \ltimes G$-equivariant:
\begin{proposition}
\label{prop:DescentMapTwist3}
Let $\ti{G}=\langle \sigma \rangle \ltimes G$.  Let $A$ be a $\ti{G}$-$C^\ast$ algebra.  Let 
\[\iota_G \colon G\ltimes A \rightarrow \ti{G} \ltimes A\] 
denote the $^\ast$-homomorphism induced by the open inclusion $G \hookrightarrow \ti{G}$.  Then
\[ \tau^A_\sigma \otimes \iota_G = \iota_G \in \KK(G\ltimes A,\ti{G} \ltimes A).\]
\end{proposition}
\begin{proof}
The automorphism $\tau^A_\sigma \in \Aut(G\ltimes A)$ extends to an automorphism of $\ti{G} \ltimes A$, given by the same formula with $\ti{g}^\sigma:=\Ad_\sigma(\ti{g}) \in \ti{G}$.  Thus
\[ \iota_G \circ \tau^A_\sigma=\ti{\tau}^A_\sigma \circ \iota_G,\]
where $\ti{\tau}^A_\sigma$ denotes the extended automorphism.  In fact $\ti{\tau}^A_\sigma$ is an inner automorphism, i.e. there exists a unitary $u_\sigma$ in the multiplier algebra of $\ti{G} \ltimes A$ such that $\ti{\tau}^A_\sigma=\Ad_{u_\sigma}$, cf. \cite[II.10.3.10]{BlackadarCAlgebras}.  The result follows from the fact that any inner automorphism of a $C^\ast$ algebra $D$ induces the identity element in $\KK(D,D)$.
\end{proof}

\section{Schr\"odinger-type operators.}\label{sec:SchrodingerOps}
If $A$ is a self-adjoint operator on a Hilbert space $H$ with domain $\dom(A)$ and spectrum in $[1,\infty)$, then one defines an associated positive definite quadratic form
\[ q_A(u_1,u_2)=(Au_1,u_2)\]
for all $u_1,u_2 \in \dom(A)$.  The completion of $\dom(A)$ using the inner product $q_A$ is a Hilbert space $\dom(q_A)$ which can be identified with $\dom(A^{1/2})$, and is known as the \emph{form domain} of $A$ (cf. \cite[VIII.6]{ReedSimonI}).  Given self-adjoint operators $A,B$ with spectrum in $[1,\infty)$ one writes
\[ A \ge B\]
if
\[ \dom(q_A) \subset \dom(q_B) \quad \text{and} \quad q_A(v,v)\ge q_B(v,v) \quad \forall v \in \dom(q_A)\]
(cf. \cite[XIII.2, p.85]{ReedSimonIV}).  Equivalently, $A \ge B$ if the inclusion mapping
\[ (\dom(q_A),q_A) \hookrightarrow (\dom(q_B),q_B)\]
is norm-decreasing.  It is enough to check the inequality $q_A(v,v)\ge q_B(v,v)$ on a core for $A$.  More generally if $A,B$ are self-adjoint operators with spectrum in $[-c,\infty)$ for some $c \ge 0$ then one writes $A \ge B$ if $A+c+1 \ge B+c+1$.

\begin{proposition}
\label{CompactnessBoundedBelow}
Let $\Y$ be a complete Riemannian manifold.  Let $\st{D}$ be a symmetric $1^{st}$ order elliptic differential operator with finite propagation speed acting on sections of a Hermitian vector bundle $S$, and let $H=L^2(\Y,S)$.  Let $\rho$, $V$ be continuous functions such that $\rho$ is bounded, $V$ is bounded below, and $V$ is proper on the support of $\rho$.  Let $A$ be a self-adjoint operator with spectrum in $(0,\infty)$, and suppose
\[ A \ge \st{D}^2+V.\]
Then the operator $\rho A^{-1}$ is compact. 
\end{proposition}
\begin{remark}
Consider the special case when $V$ is proper and bounded below.  The operator $\st{D}^2+V$ is a Schr\"{o}dinger-type operator with potential going to infinity at infinity.  In this case, it is known that $\st{D}^2+V$ has discrete spectrum.  Many proofs of this appear in the literature, cf. \cite{ShubinSchrodinger} (for $\st{D}^2=-\Delta$), \cite{KucerovskyCallias} (for a proof based on a method of Gromov-Lawson).  It is also closely related to a Fredholm criterion of Anghel \cite{Anghel1993}, and to the property of being `$\kappa$-coercive' for all $\kappa>0$ in \cite[Corollary 5.6]{BarBallmannGuide}.
\end{remark}
\begin{proof}
Note that if $\rho(A+c)^{-1}$ is compact for some $c>0$, then so is $\rho A^{-1}$ since
\[ \rho A^{-1}-\rho(A+c)^{-1}=c\rho(A+c)^{-1}A^{-1}\]
and the right hand side is compact since $\rho(A+c)^{-1}$ is compact and $A^{-1}$ is bounded.  Therefore we may as well assume that $V \ge 1$ and the spectrum of $A$ is contained in $[1,\infty)$.

The operator $\st{D}^2+V$ is self-adjoint (cf. \cite[Theorem 4.6]{ChernoffSchrodinger}) and positive.  Let $q$ denote the corresponding quadratic form, defined initially on the domain of $\st{D}^2+V$ by
\[ q(u_1,u_2)=\big((\st{D}^2+V)u_1,u_2\big)\]
and then extended to the completion $H(q)=\dom(q) \subset H$ of the domain of $\st{D}^2+V$ with respect to the inner product $q$.  Since
\[ q(u,u)=\|\st{D}u\|^2+\|\sqrt{V}u\|^2\]
the Hilbert space $H(q)$ can be described more simply as
\[ H(q)=\{u \in \dom(\st{D})|\sqrt{V}u \in H\}.\]
The operator $A^{-1}$ defines a bounded linear map $H \rightarrow \dom(A)$, where the domain $\dom(A)$ of $A$ is equipped with the $A$-norm $\|u\|_A:=\|Au\|$ (cf. \cite[VIII]{ReedSimonI}).  The Cauchy-Schwartz inequality and $q_A(u,u)\ge q(u,u)$ imply that the inclusion maps
\[ \dom(A) \subset \dom(q_A) \hookrightarrow H(q) \]
are bounded.  Thus $\rho A^{-1}$ factors as a composition of bounded linear maps:
\[ H \xrightarrow{A^{-1}}\dom(A) \hookrightarrow H(q) \xrightarrow{M_\rho} H\]
where $M_\rho$ denotes the operator given by multiplication by the function $\rho$.  It therefore suffices to show that the map
\[ M_\rho \colon H(q) \rightarrow H \]
is a compact mapping.  Without loss of generality assume $|\rho|\le 1$.  For $n \in \bZ_{\ge 0}$ let 
\[ W_n=V^{-1}((-\infty,n]).\] 
By assumption $\supp(\rho)\cap W_n$ is compact.  For each $n$, let $g_n:\Y \rightarrow [0,1]$ be a bump function equal to $1$ on $W_n$ and supported in $W_{n+1}$.  Since $g_n\rho$ has compact support, 
\[ M_{g_n\rho}\colon H(q) \rightarrow H\]
is compact, by the Rellich lemma.  We claim that the difference
\[ M_\rho-M_{g_n\rho}=M_{(1-g_n)\rho} \]
converges to zero in the norm topology on the space $\bB(H(q),H)$ of bounded operators $H(q) \rightarrow H$, which implies that $M_\rho$ is also compact.  Indeed, if $q(u,u)\le 1$ then
\[ 1 \ge q(u,u) \ge \int_{\Y \setminus W_n} n|u|^2 \qquad \Rightarrow \qquad \int_{\Y \setminus W_n}|u|^2 \le \tfrac{1}{n}.\]
Thus
\[ \|M_{(1-g_n)\rho}\|^2=\sup_{q(u,u)\le 1} \, \, \int_{\Y} (1-g_n)^2|\rho u|^2 \le \sup_{q(u,u)\le 1} \, \, \int_{\Y \setminus W_{n}}|u|^2 \le \tfrac{1}{n},\]
which goes to $0$ as $n \rightarrow \infty$.
\end{proof}

\begin{remark}
\label{CompactIsotypicComponents}
More generally, suppose a compact group $K$ acts isometrically on $\Y$, commuting with $A$, $\st{D}$ and that $V$ is $K$-invariant.  The Hilbert space $H$ decomposes into isotypic components
\[ H=\bigoplus_{\pi \in \Irr(K)} H_\pi,\]
and the operator $A$ decomposes accordingly into self-adjoint operators $A_\pi$ with $\dom(A_\pi)=\dom(A)\cap H_\pi \subset H_\pi$, and similarly for $\st{D}^2+V$.  Suppose the conditions in Proposition \ref{CompactnessBoundedBelow} are satisfied except that the inequality holds only on $H_\pi$, that is
\[ A_\pi \ge (\st{D}^2+V)_\pi. \]
Essentially the same argument, with $\rho A_\pi^{-1}$ factored as
\[ H_\pi \xrightarrow{A_\pi^{-1}} \dom(A_\pi) \hookrightarrow 
H(q)\xrightarrow{M_\rho} H\]
shows that $\rho A_\pi^{-1}:H_\pi \rightarrow H$ is a compact operator.
\end{remark}

\bibliographystyle{amsplain} 
\bibliography{../Biblio}
\end{document}